\documentclass[12pt]{article}

\usepackage{amsmath}
\usepackage{amssymb}
\usepackage{amsthm}
\usepackage{amsfonts}
\usepackage{latexsym}
\usepackage{cite}
\usepackage{psfrag}
\usepackage{epsfig}
\usepackage{graphicx}
\usepackage{mathrsfs}
\usepackage{url}
\usepackage{color}
\usepackage{cancel}
\usepackage{eufrak}
\usepackage{multirow}

\usepackage[colorlinks=true]{hyperref}
\hypersetup{urlcolor=blue, citecolor=red}

\makeatletter
\@addtoreset{equation}{section}
\makeatother

\newtheorem{thm}{Theorem}[section]
\newtheorem{lem}[thm]{Lemma}
\newtheorem{corollary}[thm]{Corollary}
\newtheorem{conjecture}[thm]{Conjecture}
\theoremstyle{definition}

\newtheorem{notation}[thm]{Notation}

\def\CC{\mathcal{C}}

\def\O{\mathcal{O}}
\def\P{\mathcal{P}}
\def\R{\mathcal{R}}
\def\T{\mathcal{T}}

\def\C{\mathscr{C}}
\def\Ll{\mathscr{L}}
\def\M{\mathscr{M}}
\def\N{\mathscr{N}}

\def\Os{\mathscr{O}}

\def\PG{\mathrm{PG}}
\def\RC{\mathrm{RC}}
\def\RA{\mathrm{RA}}
\def\Tr{\mathrm{T}}
\def\IC{\mathrm{IC}}
\def\IA{\mathrm{IA}}
\def\UG{\mathrm{U\Gamma}}
\def\UnG{\mathrm{Un\Gamma}}
\def\EG{\mathrm{E\Gamma}}
\def\EnG{\mathrm{En\Gamma}}
\def\Ar{\mathrm{A}}
\def\EA{\mathrm{EA}}

\def\MM{\mathbf{M}}
\def\Pf{\mathbf{P}}
\def\D{\mathbf{D}}

\def\F{\mathbb{F}}
\def\Ob{\mathbb{O}}

\def\A{\mathfrak{A}}
\def\Lk{\mathfrak{L}}
\def\Pk{\mathfrak{P}}
\def\pk{\mathfrak{p}}

\def\t{\text}
\def\db{\displaybreak[3]}
\def\dbn{\displaybreak[3]\notag}
\def\nt{\notag}

\begin{document}
\title{Twisted cubic and orbits of lines in $\mathrm{PG}(3,q)$
\date{}
}
\maketitle
\begin{center}
{\sc Alexander A. Davydov
\footnote{A.A. Davydov ORCID \url{https://orcid.org/0000-0002-5827-4560}}
}\\
{\sc\small Institute for Information Transmission Problems (Kharkevich institute)}\\
 {\sc\small Russian Academy of Sciences}\\
 {\sc\small Moscow, 127051, Russian Federation}\\\emph {E-mail address:} adav@iitp.ru\medskip\\
 {\sc Stefano Marcugini
 \footnote{S. Marcugini ORCID \url{https://orcid.org/0000-0002-7961-0260}},
 Fernanda Pambianco
 \footnote{F. Pambianco ORCID \url{https://orcid.org/0000-0001-5476-5365}}
 }\\
 {\sc\small Department of  Mathematics  and Computer Science,  Perugia University,}\\
 {\sc\small Perugia, 06123, Italy}\\
 \emph{E-mail address:} \{stefano.marcugini, fernanda.pambianco\}@unipg.it
\end{center}

\textbf{Abstract.}
In the projective space $\mathrm{PG}(3,q)$, we consider the orbits of lines under the stabilizer group of the twisted cubic.  It is well known that the lines can be partitioned into classes every of which is a union of line orbits.  All types of lines forming a unique orbit are found. For the rest of the line types (apart from one of them) it is proved that they form exactly two or three orbits; sizes and structures of these orbits are determined.  Problems remaining open for one type of lines are formulated. For  $5\le q\le37$ and $q=64$, they are solved.



\section{Introduction}
Let $\PG(N,q)$ be the $N$-dimensional projective space over the Galois field~$\F_q$ with $q$ elements.
An $n$-arc in  $\PG(N,q)$, with $n\ge N + 1\ge3$, is a set of $n$ points such that no $N +1$ points belong to
the same hyperplane of $\PG(N,q)$. An $n$-arc is complete if it is not contained in an $(n+1)$-arc, see \cite{BallLavrauw} and the references therein.  For an introduction to projective geometry over finite fields see \cite{Hirs_PGFF,HirsStor-2001,HirsThas-2015}.

In $\PG(N,q)$, $2\le N\le q-2$, a normal rational curve is any $(q+1)$-arc projectively equivalent to the arc
$\{(t^N,t^{N-1},\ldots,t^2,t,1):t\in \F_q\}\cup \{(1,0,\ldots,0)\}$. In $\PG(3,q)$, the normal rational curve is called a  \emph{twisted cubic} \cite{Hirs_PG3q,HirsThas-2015}. Twisted cubics have important connections with a number of other combinatorial objects.This prompted the  twisted cubics to be widely studied, see e.g. \cite{BDMP-TwCub,BonPolvTwCub,BrHirsTwCub,CLPolvT_Spr,CasseGlynn82,CasseGlynn84,CosHirsStTwCub,GiulVincTwCub,Hirs_PG3q,HirsStor-2001,HirsThas-2015,%
LunarPolv,ZanZuan2010} and the references therein.
In \cite{Hirs_PG3q}, the orbits of planes, lines and points under the group of the projectivities fixing the twisted cubic are considered. Also, in \cite{BDMP-TwCub}, the structure of the \emph{point-plane} incidence matrix of $\PG(3,q)$ using orbits under the stabilizer group of the twisted cubic is described.

\emph{In this paper}, we consider the orbits of lines in $\PG(3,q)$ under the stabilizer group $G_q$ of the twisted cubic. We use the partitions of lines into   unions of orbits (called \emph{classes}) under $G_q$ described in \cite{Hirs_PG3q}.   All types of lines forming a unique orbit are found. For the rest of the line types (apart from one of them) it is proved that they form exactly two or three orbits; sizes and structures of these orbits are determined.  Problems remaining open for one type of lines are formulated. For  $5\le q\le37$ and $q=64$, they are solved.

The theoretic results hold for $q\ge5$. For $q = 2,3,4$ we describe the orbits by computer search.

The results obtained increase our knowledge on the properties of lines in $\PG(3,q)$. The new results can be useful for feature investigations, in particular, for considerations of the plane-line incidence matrix of $\PG(3,q)$, see \cite{DMP_IPlLineIncarX}.

The paper is organized as follows. Section \ref{sec_prelimin} contains preliminaries. In Section \ref{sec_mainres}, the main results of this paper are summarized.  In Sections \ref{sec:nul_pol}--\ref{sec:orbEA}, orbits of lines in $\PG(3,q)$ under $G_q$ are considered. In Section \ref{sec:classification}, the open problems are formulated and their solutions for $5\le q\le37$ and $q=64$ are considered.

\section{Preliminaries on the twisted cubic in\\ $\PG(3,q)$}\label{sec_prelimin}
We summarize the results on the twisted cubic of \cite{Hirs_PG3q} useful in this paper.

Let $\Pf(x_0,x_1,x_2,x_3)$ be a point of $\PG(3,q)$ with homogeneous coordinates $x_i\in\F_{q}$.
 Let $\F_{q}^*=\F_{q}\setminus\{0\}$, $\F_q^+=\F_q\cup\{\infty\}$. For
For $t\in\F_q^+$, let  $P(t)$ be a point such that
\begin{align}\label{eq2:P(t)}
  P(t)=\Pf(t^3,t^2,t,1)\text{ if }t\in\F_q;~~P(\infty)=\Pf(1,0,0,0).
\end{align}

Let $\C\subset\PG(3,q)$ be the \emph{twisted cubic} consisting of $q+1$ points\\
 $P_1,\ldots,P_{q+1}$  no four of which are coplanar.
We consider $\C$ in the canonical form
\begin{align}\label{eq2_cubic}
&\C=\{P_1,P_2,\ldots,P_{q+1}\}=\{P(t)\,|\,t\in\F_q^+\}.
\end{align}

A \emph{chord} of $\C$ is a line through a pair of real points of $\C$ or a pair of complex conjugate points. In the last  case it is an \emph{imaginary chord}. If the real points are distinct, it is a \emph{real chord}; if they coincide with each other, it is a \emph{tangent.}

Let $\boldsymbol{\pi}(c_0,c_1,c_2,c_3)$ $\subset\PG(3,q)$, be the plane with equation
\begin{align}\label{eq2_plane}
  c_0x_0+c_1x_1+c_2x_2+c_3x_3=0,~c_i\in\F_q.
\end{align}
The \emph{osculating plane} in the  point $P(t)\in\C$ is as follows:
\begin{align}\label{eq2_osc_plane}
&\pi_\t{osc}(t)=\boldsymbol{\pi}(1,-3t,3t^2,-t^3)\t{ if }t\in\F_q; ~\pi_\t{osc}(\infty)=\boldsymbol{\pi}(0,0,0,1).
\end{align}
 The $q+1$ osculating planes form the osculating developable $\Gamma$ to $\C$, that is a \emph{pencil of planes} for $q\equiv0\pmod3$ or a \emph{cubic developable} for $q\not\equiv0\pmod3$.

 An \emph{axis} of $\Gamma$ is a line of $\PG(3,q)$ which is the intersection of a pair of real planes or complex conjugate planes of $\Gamma$. In the last  case it is a \emph{generator} or an \emph{imaginary axis}. If the real planes are distinct it is a \emph{real axis}; if they coincide with each other it is a \emph{tangent} to $\C$.

For $q\not\equiv0\pmod3$, the null polarity $\A$ \cite[Sections 2.1.5, 5.3]{Hirs_PGFF}, \cite[Theorem~21.1.2]{Hirs_PG3q} is given by
\begin{align}\label{eq2_null_pol}
&\Pf(x_0,x_1,x_2,x_3)\A=\boldsymbol{\pi}(x_3,-3x_2,3x_1,-x_0).
\end{align}

\begin{notation}\label{notation_1}
In future, we consider $q\equiv\xi\pmod3$ with $\xi\in\{-1,0,1\}$. Many values depend of $\xi$ or have sense only for specific $\xi$.
We note this by remarks or by superscripts ``$(\xi)$''.
 If a value is the same for all $q$ or a property holds for all $q$, or it is not relevant, or it is clear by the context, the remarks and superscripts ``$(\xi)$'' are not used. If a value is the same for $\xi=-1,1$, then one may use ``$\ne0$''. In superscripts, instead of ``$\bullet$'', one can write ``$\mathrm{od}$'' for odd $q$ or ``$\mathrm{ev}$'' for even $q$. If a value is the same for odd and even $q$, one may omit ``$\bullet$''.

The following notation is used.
\begin{align*}
  &G_q && \t{the group of projectivities in } \PG(3,q) \t{ fixing }\C;\db  \\
  &\mathbf{Z}_n&&\t{cyclic group of order }n;\db  \\
  &\mathbf{S}_n&&\t{symmetric group of degree }n;\db  \\
&A^{tr}&&\t{the transposed matrix of }A;\db \\
&\#S&&\t{the cardinality of a set }S;\db\\
&\overline{AB}&&\t{the line through the points $A$ and }B;\db\\
&\triangleq&&\t{the sign ``equality by definition"}.\db\\
&&&\t{\textbf{Types $\pi$ of planes:}}\db\\
&\Gamma\t{-plane}  &&\t{an osculating plane of }\Gamma;\db \\
&d_\C\t{-plane}&&\t{a plane containing \emph{exactly} $d$ distinct points of }\C,~d=0,2,3;\db \\
&\overline{1_\C}\t{-plane}&&\t{a plane not in $\Gamma$ containing \emph{exactly} 1 point of }\C;\db \\
&\Pk&&\t{the list of possible types $\pi$ of planes},~\Pk\triangleq\{\Gamma,2_\C,3_\C,\overline{1_\C},0_\C\};\db\\
&\pi\t{-plane}&&\t{a plane of the type }\pi\in\Pk; \db\\
&\N_\pi&&\t{the orbit of $\pi$-planes under }G_q,~\pi\in\Pk.\db\\
&&&\t{\textbf{Types $\lambda$ of lines with respect to  the twisted cubic $\C$:}}\db\\
&\RC\t{-line}&&\t{a real chord  of $\C$;}\db \\
&\RA\t{-line}&&\t{a real axis of $\Gamma$ for }\xi\ne0;\db \\
&\Tr\t{-line}&&\t{a tangent to $\C$};\db \\
&\IC\t{-line}&&\t{an imaginary chord  of $\C$;}\db \\
&\IA\t{-line}&&\t{an imaginary axis of $\Gamma$ for }\xi\ne0;\db \\
&\UG&&\t{a non-tangent unisecant in a $\Gamma$-plane;}\db \\
&\t{Un$\Gamma$-line}&&\t{a unisecant not in a $\Gamma$-plane (it is always non-tangent);}\db \\
&\t{E$\Gamma$-line}&&\t{an external line in a $\Gamma$-plane (it cannot be a chord);}\db \\
&\t{En$\Gamma$-line}&&\t{an external line, other than a chord, not in a $\Gamma$-plane;}\db \\
&\Ar\t{-line}&&\t{the axis of $\Gamma$ for }\xi=0\db\\
&&&\t{(it is the single line of intersection of all the $q+1~\Gamma$-planes)};\db \\
&\EA\t{-line}&&\t{an external line meeting the axis of $\Gamma$ for }\xi=0;\db\\
&\Lk^{(\xi)}&&\t{the list of possible types $\lambda$ of lines},\db\\
&&&\Lk^{(\ne0)}\triangleq\{\RC,\RA,\Tr,\IC,\IA,\UG,\UnG,\EG,\EnG\}\t{ for }\xi\ne0,\db\\
&&&\Lk^{(0)}\triangleq\{\RC,\Tr,\IC,\UG,\UnG, \EnG,\Ar,\EA\}\t{ for }\xi=0;\db\\
&\lambda\t{-line}&&\t{a line of the type }\lambda\in\Lk^{(\xi)};\db\\
&L_\Sigma^{(\xi)}&&\t{the total number of orbits of lines in }PG(3,q);\db\\
&L_{\lambda\Sigma}^{(\xi)\bullet}&&\t{the total number of orbits of $\lambda$-lines},~\lambda\in\Lk^{(\xi)};\db\\
&\O_\lambda&&\t{the union (class) of all $L_{\lambda\Sigma}^{(\xi)\bullet}$ orbits of $\lambda$-lines under }G_q,~\lambda\in\Lk^{(\xi)}.\db\\
&&&\t{\textbf{Types of points with respect to the twisted cubic $\C$:}}\db\\
&\C\t{-point}&&\t{a point  of }\C;\db\\
&\mu_\Gamma\t{-point}&&\t{a point  off $\C$ lying on \emph{exactly} $\mu$ distinct osculating planes,}\db\\
&&&\mu_\Gamma\in\{0_\Gamma,1_\Gamma,3_\Gamma\}\t{ for }\xi\ne0,~\mu_\Gamma\in\{(q+1)_\Gamma\}\t{ for }\xi=0;\db\\
&\Tr\t{-point}&&\t{a point  off $\C$  on a tangent to $\C$ for }\xi\ne0;\db\\
&\t{TO-point}&&\t{a point  off $\C$ on a tangent and one osculating plane for }\xi=0;\db\\
&\RC\t{-point}&&\t{a point  off $\C$  on a real chord;}\db\\
&\IC\t{-point}&&\t{a point  on an imaginary chord (it is always off $\C$).}
\end{align*}
\end{notation}

The following theorem summarizes results from \cite{Hirs_PG3q} useful in this paper.
\begin{thm}\label{th2_Hirs}
\emph{\cite[Chapter 21]{Hirs_PG3q}} The following properties of the twisted cubic $\C$ of \eqref{eq2_cubic} hold:
  \begin{align}
  &\textbf{\emph{(i)}} \t{ The group $G_q$ acts triply transitively on $\C$. Also,}\dbn\\
  &\t{\textbf{\emph{(a)}}}~~ G_q\cong PGL(2,q)~\t{ for }q\ge5;\dbn \\
 &\phantom{\t{\textbf{\emph{(a)}}}~~} G_4\cong\mathbf{S}_5\cong P\Gamma L(2,4)\cong\mathbf{Z}_2PGL(2,4),~\#G_4=2\cdot\#PGL(2,4)=120;\dbn \\
 &\phantom{\t{\textbf{\emph{(a)}}}~~} G_3\cong\mathbf{S}_4\mathbf{Z}_2^3,\hspace{4.2cm}\#G_3=8\cdot\#PGL(2,3)=192;\dbn \\
 &\phantom{\t{\textbf{\emph{(a)}}}~~} G_2\cong\mathbf{S}_3\mathbf{Z}_2^3,\hspace{4.2cm}\#G_2=8\cdot\#PGL(2,2)=48.\dbn\\
&\textbf{\emph{(b)}} \t{ The matrix $\MM$ corresponding to a projectivity of $G_q$ has the general}\dbn\\
&\phantom{\textbf{\emph{(b)}}}\t{ form}\dbn\\
& \label{eq2_M} \mathbf{M}=\left[
 \begin{array}{cccc}
 a^3&a^2c&ac^2&c^3\\
 3a^2b&a^2d+2abc&bc^2+2acd&3c^2d\\
 3ab^2&b^2c+2abd&ad^2+2bcd&3cd^2\\
 b^3&b^2d&bd^2&d^3
 \end{array}
  \right],~a,b,c,d\in\F_q,\db\\
 & ad-bc\ne0.\nt
\end{align}

\textbf{\emph{(ii)}} Under $G_q$, $q\ge5$, there are five orbits of planes and five orbits of points.

\textbf{\emph{(a)}} For all $q$, the orbits $\N_i$ of planes are as follows:
\begin{align}\label{eq2_plane orbit_gen}
   &\N_1=\N_\Gamma=\{\Gamma\t{-planes}\},~~~~\#\N_\Gamma=q+1;\db\\
   &\N_{2}=\N_{2_\C}=\{2_\C\t{-planes}\}, ~\#\N_{2_\C}=q^2+q;\dbn \\
 &\N_{3}=\N_{3_\C}=\{3_\C\t{-planes}\},~  \#\N_{3_\C}=(q^3-q)/6;\dbn\\
 &\N_{4}=\N_{\overline{1_\C}}=\{\overline{1_\C}\t{-planes}\},~\#\N_{\overline{1_\C}}=(q^3-q)/2;\dbn\\
 &\N_{5}=\N_{0_\C}=\{0_\C\t{-planes}\},~\#\N_{0_\C}=(q^3-q)/3.\nt
 \end{align}

\textbf{\emph{(b)}} For $q\not\equiv0\pmod 3$, the orbits $\M_j$ of points are as follows:
\begin{align*}
&\M_1=\{\C\t{-points}\},~ \M_2=\{\Tr\t{-points}\},~\M_3=\{3_\Gamma\t{-points}\},,\db\\
&\M_4=\{1_\Gamma\t{-points}\},~\M_5=\{0_\Gamma\t{-points}\}.\db\\
&\t{Also, if } q\equiv1\pmod 3 \t{ then } \M_{3}\cup\M_{5}=\{\RC\t{-points}\}, ~\M_{4}=\{\IC\t{-points}\};\db\\
 &\t{if } q\equiv-1\pmod 3\t{ then }\M_{3}\cup\M_{5}=\{\IC\t{-points}\},~ \M_{4}=\{\RC\t{-points}\}.
\end{align*}

\textbf{\emph{(c)}} For $q\equiv0\pmod 3$, the orbits $\M_j$ of points are as follows:
\begin{align*}
&\M_1=\{\C\t{-points}\},~\M_2=\{(q+1)_\Gamma\t{-points}\},~\M_3=\{\t{\emph{TO}-points}\},\db\\
&\M_4=\{\RC\t{-points}\},~ \M_5=\{\IC\t{-points}\}.
\end{align*}

 \textbf{\emph{(iii)}} For $q\not\equiv0\pmod3$, the null polarity $\A$ \eqref{eq2_null_pol} interchanges $\C$ and $\Gamma$ and their corresponding chords and axes.

 \textbf{\emph{(iv)}} The lines of $\PG(3,q)$ can be partitioned into classes called $\O_i$ and $\O'_i$, each of which is a union of orbits under $G_q$.
  \begin{align}
  &\hspace{0.6cm}\textbf{\emph{(a)}}~ q\not\equiv0\pmod3,~ q\ge5, ~\O'_i=\O_i\A,~ \#\O'_i=\#\O_i,~i=1,\ldots,6.\dbn\\
  &\O_1=\O_\RC=\{\RC\t{-lines}\},~\O'_1=\O_\RA=\{\RA\t{-lines}\},\db\label{eq2_classes line q!=0mod3}\\
  &\#\O_\RC=\#\O_\RA=(q^2+q)/2;\dbn\\
  &\O_2=\O'_2=\O_\Tr=\{\Tr\t{-lines}\},~\#\O_\Tr=q+1;\dbn \\
  &\O_3=\O_\IC=\{\IC\t{-lines}\},~\O'_3=\O_\IA=\{\IA\t{-lines}\},\dbn\\
  &\#\O_\IC=\#\O_\IA=(q^2-q)/2;\dbn\\
  &\O_4=\O'_4=\O_\UG=\{\UG\t{-lines}\},~\#\O_\UG=q^2+q;\dbn\\
  &\O_5=\O_\UnG=\{\UnG\t{-lines}\},\O'_5=\O_\EG=\{\EG\t{-lines}\},\dbn\\
  &\#\O_\UnG=\#\O_\EG=q^3-q;\dbn\\
  &\O_6=\O'_6=\O_\EnG=\{\EnG\t{-lines}\},~\#\O_\EnG=(q^2-q)(q^2-1).\nt
     \end{align}
  For $q>4$ even, the lines in the regulus complementary to that of the tangents form an orbit of size $q+1$ contained in $\O_4=\O_\UG$.
  \begin{align}
  &\textbf{\emph{(b)}}~q\equiv0\pmod3,~q>3.\dbn\\
  &\t{Classes }\O_1,\ldots,\O_6\t{ are as in \eqref{eq2_classes line q!=0mod3}};~\O_7=\O_\Ar=\{\Ar\t{-line}\},~\#\O_\Ar=1;\label{eq2_classes line q=0mod3}\\
  &\O_8=\O_\EA=\{\EA\t{-lines}\},~\#\O_\EA=(q+1)(q^2-1). \nt
     \end{align}

 \textbf{\emph{(v)}} The following properties of chords and axes hold.

 \textbf{\emph{(a)}}  For all $q$, no two chords of $\C$ meet off $\C$.

 \phantom{\textbf{\emph{(a)}}} Every point off $\C$ lies on exactly one chord of $\C$.

 \textbf{\emph{(b)}}       Let $q\not\equiv0\pmod3$.

 \phantom{\textbf{\emph{(b)}}}  No two axes of $\Gamma$ meet unless they lie in the same plane of $\Gamma$.

 \phantom{\textbf{\emph{(b)}}}  Every plane not in $\Gamma$ contains exactly one axis of $\Gamma$.

  \textbf{\emph{(vi)}} For $q>2$, the unisecants of $\C$ such that every plane through such a unisecant meets $\C$ in at most one point other than the point of contact are, for $q$ odd, the tangents, while for $q$ even, the tangents and the unisecants in the complementary regulus.
\end{thm}

\section{The main results}\label{sec_mainres}
Throughout the paper, we consider orbits of lines and planes under $G_q$.

From now on, we consider $q\ge5$ apart from Theorem \ref{th3:q=2 3 4}.

Theorem \ref{th3_main_res} summarizes the results of  Sections \ref{sec:nul_pol}--\ref{sec:orbEA}.

\begin{thm}\label{th3_main_res}
Let $q\ge5$, $q\equiv\xi\pmod3$. Let notations be as in Section \emph{\ref{sec_prelimin}} including  Notation~\emph{\ref{notation_1}}. For line orbits under $G_q$ the following holds.
\begin{description}
  \item[(i)] The following classes of lines consist of a single orbit:\\
   $\O_1=\O_\RC=\{\RC\t{-lines}\}$,
  $\O_2=\O_\Tr=\{\Tr\t{-lines}\}$, and\\
   $\O_3=\O_\IC=\{\IC\t{-lines}\}$,  for all~$q$;\\
   $\O_4=\O_\UG=\{\UG\t{-lines}\}$, for odd $q$;\\
    $\O_5=\O_\UnG=\{\UnG\t{-lines}\}$ and $\O'_5=\O_\EG=\{\EG\t{-lines}\}$, for even $q$;\\
     $\O_1'=\O_\RA=\{\RA\t{-lines}\}$ and $\O_3'=\O_\IA=\{\IA\t{-lines}\}$,
  for $\xi\ne0$;\\
   $\O_7=\O_\Ar=\{\Ar\t{-lines}\}$, for $\xi=0$.

  \item[(ii)]
Let $q\ge8$ be even. The non-tangent unisecants in a $\Gamma$-plane \emph{(}i.e. $\UG$-lines, class $\O_4=\O_\UG$\emph{)} form two orbits of size $q+1$ and $q^2-1$. The orbit of size $q+1$  consists of the lines in the regulus complementary to that of the tangents. Also, the $(q+1)$-orbit and $(q^2-1)$-orbit can be represented in the form $\{\ell_1\varphi|\varphi\in G_q\}$ and $\{\ell_2\varphi|\varphi\in G_q\}$, respectively, where $\ell_j$ is a line such that  $\ell_1=\overline{P_0\Pf(0,1,0,0)}$, $\ell_2=\overline{P_0\Pf(0,1,1,0)}$, $P_0=\Pf(0,0,0,1)\in\C$.

  \item[(iii)] Let $q\ge5$ be odd.
 The non-tangent unisecants not in a $\Gamma$-plane \emph{(}i.e. $\UnG$-lines, class $\O_5=\O_\UnG$\emph{)} form  two orbits of size $\frac{1}{2}(q^3-q)$. These orbits can be represented in the form $\{\ell_j\varphi|\varphi\in G_q\}$, $j=1,2$, where $\ell_j$ is a line such that $\ell_1=\overline{P_0\Pf(1,0,1,0)}$,  $\ell_2=\overline{P_0\Pf(1,0,\rho,0)}$, $P_0=\Pf(0,0,0,1)\in\C$, $\rho$ is not a square.

   \item[(iv)] \looseness -1
Let $q\ge5$ be odd. Let $q\not\equiv0\pmod 3$. The external lines in a $\Gamma$-plane   \emph{(}class $\O_5'=\O_\EG$\emph{)} form two orbits of size $(q^3-q)/2$. These orbits can be represented in the form $\{\ell_j\varphi|\varphi\in G_q\}$, $j=1,2$, where $\ell_j=\pk_0\cap\pk_j$ is the intersection line of planes $\pk_0$ and $\pk_j$ such that
           $\pk_0=\boldsymbol{\pi}(1,0,0,0)=\pi_\t{\emph{osc}}(0)$, $\pk_1=\boldsymbol{\pi}(0,-3,0,-1)$,  $\pk_2=\boldsymbol{\pi}(0,-3\rho,0,-1)$, $\rho$ is not a square, cf. \eqref{eq2_plane}, \eqref{eq2_osc_plane}.

   \item[(v)]
Let $q\equiv0\pmod 3,\; q\ge9$. The external lines meeting the axis of $\Gamma$ \emph{(}i.e. $\EA$-lines, class $\O_8=\O_\EA$\emph{)} form three orbits of size $q^3-q$, $(q^2-1)/2$, $(q^2-1)/2$. The $(q^3-q)$-orbit and the two $(q^2-1)/2$-orbits can be represented in the form $\{\ell_1\varphi|\varphi\in G_q\}$ and $\{\ell_j\varphi|\varphi\in G_q\}$, $j=2,3$, respectively, where $\ell_j$ are lines such that $\ell_1=\overline{P_0^\Ar\Pf(0,0,1,1)}$,  $\ell_2=\overline{P_0^\Ar\Pf(1,0,1,0)}$, $\ell_3=\overline{P_0^\Ar\Pf(1,0,\rho,0)}$, $P_0^\Ar=\Pf(0,1,0,0)$, $\rho$ is not a square.
\end{description}
\end{thm}

Theorem~\ref{th3:q=2 3 4} is obtained by an exhaustive computer search using the symbol calculation system Magma~\cite{Magma}.

\begin{thm}\label{th3:q=2 3 4}
Let notations be as in Section \emph{\ref{sec_prelimin}} including  Notation~\emph{\ref{notation_1}}. For line orbits under $G_q$ the following holds.
\begin{description}
    \item[(i)] Let $q=2$. The group $G_2\cong\mathbf{S}_3\mathbf{Z}_2^3$ contains $8$ subgroups isomorphic to $PGL(2,2)$ divided into two conjugacy classes. For one of these subgroups, the matrices corresponding to the projectivities of the subgroup assume the form described by \eqref{eq2_M}. For this subgroup (and only for it) the line orbits under it are the same as in Theorem \emph{\ref{th3_main_res}} for $q\equiv-1\pmod3$.

    \item[(ii)] Let $q=3$. The group $G_3\cong\mathbf{S}_4\mathbf{Z}_2^3$ contains $24$ subgroups isomorphic to $PGL(2,3)$ divided into four conjugacy classes. For one of these subgroups, the matrices corresponding to the projectivities of the subgroup assume the form described by \eqref{eq2_M}. For this subgroup (and only for it) the line orbits under it are the same as in Theorem \emph{\ref{th3_main_res}} for $q\equiv0\pmod3$.

    \item[(iii)] Let $q=4$. The group $G_4\cong\mathbf{S}_5\cong P\Gamma L(2,4)$ contains one subgroup isomorphic to $PGL(2,4)$. The matrices corresponding to the projectivities of this subgroup assume the form described by \eqref{eq2_M} and for this subgroup the line orbits under it are the same as in Theorem \emph{\ref{th3_main_res}} for $q\equiv1\pmod3$.
\end{description}
\end{thm}

%\begin{rmk}
%Open problems and their solutions for $5\le q\le37$ and $q=64$ are given in Section \ref{sec:classification}.
%\end{rmk}

\section{The null polarity $\A$ and orbits under $G_q$ of lines in $\PG(3,q)$}\label{sec:nul_pol}
\begin{lem}\label{lem5:inversM}
  Let   $\MM$ be the general form of the matrix corresponding to a projectivity of $G_q$ given by \eqref{eq2_M}. Then its inverse matrix $\MM^{-1}$ has the form
  \begin{align}\label{eq5_inversM}
&    \MM^{-1}=\left[
 \begin{array}{cccc}
 d^3A^{-1}&cd^2A^{-1}&c^2 dA^{-1}&c^3A^{-1}\\
 3bd^2A^{-1}&d(ad+2bc)A^{-1}&c(2ad+bc)A^{-1}&3ac^2B^{-1}\\
 3b^2dB^{-1}&b(2ad+bc)B^{-1}&a(ad+2bc)B^{-1}&3a^2cA^{-1}\\
 b^3A^{-1}&ab^2A^{-1}&a^2bA^{-1}&a^3A^{-1}
 \end{array}
 \right],\db\\
 &A=a^3d^3-b^3 c^3+3ab^2c^2d-3a^2bcd^2,~B=(a^2d^2-2abcd+b^2c^2)(ad-bc).\nt
  \end{align}
\end{lem}

\begin{proof}
    The assertion is obtained with the help of the system of symbolic computation  Maple \cite{Maple}. Note that by \eqref{eq2_M}, we have $ad-bc\ne0$.
\end{proof}
\begin{lem}
 Let $q\not\equiv0\pmod 3$. Let $\A$ be the null polarity \emph{\cite[Theorem 21.1.2]{Hirs_PG3q}} given by \eqref{eq2_null_pol}. Let $P=\Pf(x_0,x_1,x_2,x_3)$ be a point of $\PG(3,q)$, $P\A$ be its polar plane, and $\Psi$ be a projectivity belonging to $G_q$.
    Then
  \begin{align}\label{eq5_FU=UF}
   (P\A)\Psi=(P\Psi)\A.
  \end{align}
\end{lem}
\begin{proof}Let ``$\times$'' note the matrix multiplication. Using the matrices $\MM$ and $\MM^{-1}$ of \eqref{eq2_M} and \eqref{eq5_inversM}, respectively, we define $x_i'$ and $\overline{c_i}$ as follows:\\
$
 [x_0',x_1',x_2',x_3']=[x_0,x_1,x_2,x_3]\times\mathbf{M},\,[\overline{c_0},\overline{c_1},\overline{c_2},\overline{c_3}]^{tr}=\mathbf{M}^{-1}
\times[c_0,c_1,c_2,c_3]^{tr}.
$
Then it is well known (see e.g. \cite[Chapter 4, Note 23]{Cassebook}) that:
\begin{align*}
\boldsymbol{\pi}(c_0,c_1,c_2,c_3)\Psi=\boldsymbol{\pi}(\overline{c_0},\overline{c_1},\overline{c_2},\overline{c_3}).
\end{align*}

By above and by \eqref{eq2_null_pol}, \eqref{eq2_M}, \eqref{eq5_inversM}, we have $P\Psi=\Pf(x_0',x_1',x_2',x_3');$
\begin{align*}
&(P\Psi)\A=\boldsymbol{\pi}(x_3',-3x_2',3x_1',-x_0');~~P\A=\boldsymbol{\pi}(x_3,-3x_2,3x_1,-x_0);\db \\
&(P\A)\Psi=\boldsymbol{\pi}(v_0,v_1,v_2,v_3),~[v_0,v_1,v_2,v_3]^{tr}=\mathbf{M}^{-1}\times[x_3,-3x_2,3x_1,-x_0]^{tr}.
\end{align*}
By direct symbolic computation using the system Maple, we verified that
\begin{align*}
   \mathbf{M}^{-1}\times[x_3,-3x_2,3x_1,-x_0]^{tr}=[x_3',-3x_2',3x_1',-x_0']^{tr}. \hspace{3cm} \qedhere
\end{align*}
\end{proof}

\begin{thm}\label{th5_null_pol}
 Let $q\not\equiv0\pmod 3$. Let $\Ll$ be an orbit of lines under~$G_q$. Then $\Ll\mathfrak{A}$ also is an orbit of lines under~$G_q$.
\end{thm}
\begin{proof}
 We take the line $\ell_1$ through the points $P_1$ and $P_2$ of $\PG(3,q)$ and a projectivity $\Psi\in G_q$. Let $\ell_2$ be the line through
$Q_1 = P_1\Psi$ and $Q_2 = P_2\Psi$. Then $\ell_1$ and $\ell_2$  belong to the same orbit
and $\ell_2=\ell_1\Psi$.

 We show that $\ell_2\A=(\ell_1\A)\Psi$.
Let $\pk_i=P_i\A,\,\pk_i'=Q_i\A,\, i=1,2$. By~\eqref{eq5_FU=UF},
\begin{align*}
 &\pk_1'=Q_1\A=(P_1\Psi)\A=(P_1\A)\Psi=\pk_1\Psi,\db\\
 &\pk_2'=Q_2\A=(P_2\Psi)\A=(P_2\A)\Psi=\pk_2\Psi.
\end{align*}
So,  we have  $\ell_2\A=\pk_1' \cap \pk_2' = \pk_1 \Psi \cap \pk_2 \Psi =  (\pk_1 \cap \pk_2)\Psi = (\ell_1\A)\Psi$.
\end{proof}

\section{Orbits under $G_q$ of chords of the cubic $\C$ and axes of the osculating developable $\Gamma$ (orbits of $\RC$-, $\Tr$-, $\IC$-, $\RA$-, and $\IA$-lines)}\label{sec:orbChAx}

\begin{thm}\label{th5:RC-lines}
For any $q\ge5$, the real chords \emph{(}i.e.\ $\RC$-lines, class $\O_1=\O_\RC$\emph{)} of the twisted cubic $\C$ \eqref{eq2_cubic} form an orbit under $G_q$.
\end{thm}

\begin{proof} We consider real chords $\R\CC_1=\overline{P(t_1)P(t_2)}$ and $\R\CC_2=\overline{P(t_3)P(t_4)}$ through the real points of $\C$, respectively, $P(t_1),P(t_2)$ and $P(t_3),P(t_4)$ such that $t_l\ne t_2$, $t_3\ne t_4$, $\{t_1,t_2\}\ne\{t_3,t_4\}$. The group $G_q$ acts triply transitively on $\C$, see Theorem~\ref{th2_Hirs}(i). So, there is a projectivity $\Psi\in G_q$ such that $\{P(t_1),P(t_2)\}\Psi=\{P(t_3),P(t_4)\}$. This projectivity maps also $\R\CC_1$ to $\R\CC_2$, i.e. $\R\CC_1\Psi=\R\CC_2$. So, the real chords form an orbit under $G_q$.
\end{proof}

\begin{corollary}\label{cor5:O1'}
 Let $q\not\equiv0\pmod 3$. In $\PG(3,q)$, for the osculating developable $\Gamma$ of the twisted cubic $\C$ \eqref{eq2_cubic}, the real axes \emph{(}i.e. $\RA$-lines, class $\O_1'=\O_\RA$\emph{)} form an orbit under~$G_q$.
\end{corollary}

\begin{proof}
  The assertion follows from Theorems \ref{th2_Hirs}(iv)(a), \ref{th5_null_pol}, and \ref{th5:RC-lines}.
\end{proof}

\begin{thm}\label{th5:T-lines}
For any $q\ge5$, the tangents \emph{(}i.e.\ $\Tr$-lines, class $\O_2=\O_\Tr$\emph{)} to the twisted cubic $\C$ \eqref{eq2_cubic} form an orbit under $G_q$. Moreover, the group $G_q$ acts triply transitively on this orbit.
\end{thm}

\begin{proof} We consider two tangents $\T_{t_1}=\overline{P(t_1)P(t_1)}$ and $\T_{t_2}=\overline{P(t_2)P(t_2)}$ through the real points $P(t_1),P(t_1)$ and $P(t_2),P(t_2)$ such that $t_l\ne t_2$. As the points of $\C$ form an orbit under $G_q$, there is a projectivity $\Psi\in G_q$ such that $P(t_1)\Psi=P(t_2)$. This projectivity maps also $\T_{t_1}$ to $\T_{t_2}$, i.e. $\T_{t_1}\Psi=\T_{t_2}$. Thus, the tangents form an orbit under~$G_q$. On this orbit,  $G_q$ acts triply transitively since $G_q$ acts triply transitively on $\C$.
\end{proof}

\begin{thm}\label{th5:O3orbit}
    For any $q\ge5$, in $\PG(3,q)$, the imaginary chords \emph{(}i.e.\ $\IC$-lines, class $\O_3=\O_\IC$\emph{)} of the twisted cubic $\C$ \eqref{eq2_cubic} form an orbit under~$G_q$.
\end{thm}

\begin{proof} Let $q\equiv\xi\pmod3$.
By Theorem \ref{th2_Hirs}(ii)(b)(c), for $\xi=1$ (resp.  $\xi=0$),  points on imaginary chords form the orbit $\M_4$ (resp. $\M_5$). If $\xi= -1$, points on $\IC$-lines are divided into two orbits $\mathscr{M}_3=\{\t{points on three osculating pla-}$ $\t{nes}\}$ and
    $\mathscr{M}_5=\{\t{points on no osculating plane}\}$. As in $\PG(3,q)$ a plane and a line always meet, for $\xi= -1$ every imaginary chord contains  a point belonging to an osculating plane and therefore to  $\mathscr{M}_3$.

Now, for any $q$, suppose that there exist at least two orbits $\overline{\O}_1$ and  $\overline{\O}_2$ of imaginary chords. Consider IC-lines $\ell_1 \in \overline{\O}_1$ and $\ell_2 \in \overline{\O}_2$. By Theorem \ref{th2_Hirs}(v)(a), no two chords of $\C$ meet off $\C$. Thus, $\ell_1 \cap \ell_2 = \emptyset$ and there exist at least two points  $P_1 \in \ell_1$ and $P_2 \in \ell_2$  belonging to the same orbit; it is $\M_4$, $\M_5$, and $\M_3$ for $\xi=1,0$, and $-1$, respectively. So, there is $\varphi \in G_q$ such that $P_1\varphi = P_2$. A projectivity maps a line to a line; as all points on $\IC$-lines are placed in ``own'' orbits (one or two) that do not contain points of other types, $\ell_1\varphi$ is an IC-line. Moreover, by Theorem \ref{th2_Hirs}(v)(a),  every point off $\C$ lies on exactly one chord; thus, $\ell_1\varphi$ is the only imaginary chord containing~$P_2$, i.e. $\ell_1\varphi =\ell_2$. So,  $\overline{\O}_1 = \overline{\O}_2$.
\end{proof}

\begin{corollary}\label{cor5_O3'}
    Let $q\not\equiv0\pmod 3$. In $\PG(3,q)$, for the osculating developable $\Gamma$ of the twisted cubic $\C$ \eqref{eq2_cubic}, the imaginary axes \emph{(}class $\O_3'=\O_\IA$\emph{)} form an orbit under~$G_q$.
\end{corollary}

\begin{proof}
  The assertion follows from Theorems \ref{th2_Hirs}(iv)(a), \ref{th5_null_pol}, and \ref{th5:O3orbit}.
\end{proof}

\section{Orbits under $G_q$ of non-tangent unisecants and external lines with respect to the cubic $\C$ (orbits of $\UG$-, $\UnG$-, and $\EG$-lines)}\label{sec:orbUGUnGEG}

\begin{notation}\label{notation_3}
In addition to Notation \ref{notation_1}, the following notation is used.
\begin{align*}
&P_t && \t{the point } P(t) \t{ of } \C \t{ with } t\in\F_q^+, \t{ cf. } \eqref{eq2:P(t)}, \eqref{eq2_cubic};\db  \\
&\T_t  && \t{the tangent line to } \C \t{ at the point } \P_t;\db  \\
&G_q^{P_t}&&  \t{the subgroup of } G_q \t{ fixing }P_t;\db  \\
&\Ob_{\lambda_i} &&\t{the set of lines from $\O_\lambda$ through $P_i$, i.e. }\Ob_{\lambda_i}\triangleq  \{\ell \in \O_\lambda | P_i \in \ell\}.
\end{align*}
\end{notation}

\begin{lem}\label{eqTangents}
The tangent $\T_t$ to $\C$ at the point $P_t$ has the following equation:
\begin{align*}
&  \T_{\infty} \t{ has equation } \begin{cases}
												x_2=0 \\
												x_3=0
            								\end{cases};~~~~  \T_1 \t{ has equation }\begin{cases}
												 x_0 = x_1 + x_2- x_3 \\
												x_0 = 3x_2 -2x_3
            								\end{cases};\\
& \T_t, t\in\F_q, t \neq 1, \t{ has equation } \begin{cases}
												 x_0 = tx_1 +t^2 x_2-t^3 x_3\\
												x_1 = tx_0 + (2t-3t^3) x_2 + (2t^4-t^2) x_3
            								\end{cases}.
\end{align*}
\end{lem}
\begin{proof}
The point $P_t= \Pf (t^3,t^2,t,1)$, $ t\in\F_q$,  can be considered as an affine point with respect to the infinite plane $x_3=0$. Then the slope of the tangent line to $\C$ at $P_t$ is obtained by deriving the parametric equation of $\C$ and is $(3t^2,2t,1)$. It means that $\T_t$ contains the infinite point $Q_t = \Pf (3t^2,2t,1,0)$.
The planes $\pk_1$ of equation $ x_0 = tx_1 +t^2 x_2-t^3 x_3$ and
												$\pk_2$ of equation $ x_1 = tx_0 + (2t-3t^3) x_2 + (2t^4-t^2) x_3$ contain both the points $P_t$ and $Q_t$.

However, if $t=1$,  $\pk_1 = \pk_2$, so we consider $\pk_3$ of equation $ x_0 = 3x_2 -2x_3$ as second plane containing both $P_t$ and $Q_t$.
In particular $\T_0$ has equation  $ x_0 = 0, x_1 =0 $.

Now consider the projectivity $\Psi$  of equation  $ x_0' = x_3,  x_1' = x_2,  x_2' = x_1,  x_3' = x_0$. Then
$P_0 \Psi = P_{\infty}$, $P_{\infty} \Psi = P_0$, and  $P_{t} \Psi = P_{1/t}$ if $t \neq 0$. It means that $\Psi \in G_q$ and $\T_{\infty} = \T_0 \Psi$ has equation  $ x_2 = 0, x_3 =0 $.
\end{proof}

\begin{lem}
The general form of the matrix $\MM^{P_0}$ corresponding to a projectivity of $G_q^{P_0}$ is as follows:
\begin{align}\label{eq5_M_P0}
  \MM^{P_0}=\left[
 \begin{array}{cccc}
 1&c&c^2&c^3\\
 0&d&2cd&3c^2d\\
 0&0&d^2&3cd^2\\
 0&0&0&d^3
 \end{array}
  \right],~c \in \F_q, ~d \in \F_q^*.
\end{align}
\end{lem}

\begin{proof} Let $\MM$ of \eqref{eq2_M} correspond to a projectivity $\Psi \in G_q$.
We have $[0,0,0,1]\times\MM = [b^3,b^2d,bd^2,d^3]$.
Then  $\Psi \in G_q^{P_0}$ if and only if $b = 0$, $d \neq 0$. Also, we should put $a \neq 0$, to provide $ad-bc\ne0$, see \eqref{eq2_M}.
One may choose $a = 1$, see \eqref{eq5_M_P0}, as we consider points in homogeneous coordinates.
\end{proof}

\begin{lem}\label{lem5:conj}
$G_q^{P_i}$ and  $G_q^{P_j}$ are conjugate subgroups of $G_q$, $i,j\in\F_q^+$.
\end{lem}
\begin{proof}
As $G_q$ acts transitively on $\C$, there exists  $\Psi \in G_q$ such that $P_i \Psi = P_j$. Then   $ \Psi^{-1} G_q^{P_i} \Psi =G_q^{P_j}$. In fact, let $\varphi \in G_q^{P_i}$. Then  $P_j \Psi^{-1} \varphi \Psi = P_i \varphi \Psi = P_i \Psi = P_j$.
On the other hand, let $\gamma \in G_q^{P_j}$. Then  $P_i \Psi \gamma \Psi^{-1} = P_j  \gamma \Psi^{-1} = P_j  \Psi^{-1}= P_i$. It means that $\Psi \gamma \Psi^{-1} \in G_q^{P_i}$, i.e. $\gamma \in \Psi^{-1} G_q^{P_i} \Psi$.
\end{proof}

\begin{corollary}\label{cor5:size}
For all $t\in\F_q^+$, we have $\#G_q^{P_t}= q(q-1)$.
\end{corollary}

\begin{proof}
By \eqref{eq5_M_P0}, $\#G_q^{P_0} = q(q-1)$. By Lemma \ref{lem5:conj}, there exists  $\Psi \in G_q$ such that  $ \Psi^{-1} G_q^{P_0} \Psi =G_q^{P_t}$,  $t\in\F_q^+$. Then $G_q^{P_0} \Psi =  \Psi  G_q^{P_t}$. As a finite group and its cosets have the same cardinality,
%\cite[Lemma 2.4.5]{Herstein}
 $\#G_q^{P_0}=\#G_q^{P_0}\Psi=\#\Psi  G_q^{P_t}=\# G_q^{P_t}$.
\end{proof}

\begin{lem}\label{lem5:im_cont}
Let $\lambda \in \{\RC,\Tr,\UG,\UnG\}$. Then $\Ob_{\lambda_i} G_q^{P_i}  = \Ob_{\lambda_i}$.
\end{lem}
\begin{proof}
Let $\ell \in \Ob_{\lambda_i}, \varphi \in G_q^{P_i}$.  As $P_i \in \ell$,  $P_i \varphi = P_i \in \ell \varphi$. For $\lambda \in \{\RC,\Tr,\UG,$ $\UnG\}$,  $\ell$ of type $\lambda$ implies $\ell \varphi$ of type $\lambda$. Therefore, $\ell \varphi \in  \Ob_{\lambda_i}$. On the other hand, if $I$ is the identity element of $G_q^{P_i}$, $\Ob_{\lambda_i} G_q^{P_i} \supseteq \Ob_{\lambda_i} I = \Ob_{\lambda_i}$.
\end{proof}

\begin{lem}\label{lem5:eq_im}
Let $\ell$ be a line such that $P_i \in \ell$. Let $\Os_{\ell} =   \left\{\ell \varphi | \varphi \in G_q^{P_i}\right\}$, $\Psi_1, \Psi_2 \in G_q$.
If  $P_i \Psi_1 = P_i \Psi_2 = P_j$ then $\Os_{\ell}\Psi_1  = \Os_{\ell}\Psi_2$.
\end{lem}
\begin{proof}
As $  P_i \Psi_1 \Psi_2^{-1} = P_j \Psi_2^{-1}=  P_i$, we have $\Psi_1 \Psi_2^{-1} \in G_q^{P_i}$.
Let $\overline{\ell} \in \Os_{\ell}\Psi_1$. Then $\overline{\ell} =  \ell \varphi \Psi_1$,  $\varphi \in G_q^{P_i}$. This implies $ \overline{\ell}  \Psi_2^{-1} =  \ell \varphi \Psi_1  \Psi_2^{-1} \in \Os_{\ell}$, whence   $\overline{\ell}  \in \Os_{\ell}\Psi_2$. The proof of the other inclusion is analogous.
\end{proof}

\begin{lem}\label{lem5:im_diff}
Let $\lambda \in \{\UG,\UnG\}$, $\ell_1, \ell_2, \in \Ob_{\lambda_i}$,  $\Os_{\ell_1} =   \left\{\ell_1 \varphi | \varphi \in G_q^{P_i}\right\}$,  $\Os_{\ell_2} =   \left\{\ell_2 \varphi | \varphi \in G_q^{P_i}\right\}$. If  $\Os_{\ell_1} \cap \Os_{\ell_2} = \emptyset$ then $\Os_{\ell_1} G_q \cap \Os_{\ell_2} G_q = \emptyset$.
\end{lem}
\begin{proof}
Suppose  $\overline{\ell}  \in \Os_{\ell_1} G_q \cap \Os_{\ell_2} G_q$. Then  $\overline{\ell}$ is a line of the same type $\lambda$ as  $\ell_1$ and $\ell_2$, i.e. it is a unisecant of $\C$, so there exists $P_j$ such that $P_j \in \overline{\ell}$.
As $\overline{\ell} \in \Os_{\ell_1} G_q$, $\overline{\ell} = \ell_1 \varphi_1 \Psi_1$,   $\varphi_1 \in G_q^{P_i}$,  $\Psi_1 \in G_q$. The point $P_i$ belongs to the line $\ell'= \ell_1 \varphi_1$. As $P_j \in \overline{\ell} = \ell' \Psi_1$ and $\ell_1, \ell', \overline{\ell} $ are unisecants and $\Psi_1 \in  G_q $, $P_i \Psi_1$ is the only point of $ \overline{\ell}$ belonging to $\C$, i.e. $P_i \Psi_1 = P_j$. Analogously, $\overline{\ell} \in \Os_{\ell_2} G_q$ implies $\overline{\ell} = \ell_2 \varphi_2 \Psi_2$,   $\varphi_2 \in G_q^{P_i}$,  $P_i \Psi_2 = P_j$.
Then $P_i \Psi_1 \Psi_2^{-1} = P_j \Psi_2^{-1}= P_i$, that implies $\Psi_1 \Psi_2^{-1} \in G_q^{P_i}$. Finally, $\ell_1 \varphi_1 \Psi_1 = \ell_2 \varphi_2 \Psi_2$ implies $\ell_1 \varphi_1 \Psi_1 \Psi_2^{-1} \varphi_2^{-1} = \ell_2$, whence  $\ell_2 \in  \Os_{\ell_1}$.
\end{proof}

\begin{lem}\label{lem5:cardinality}
Let $\lambda \in \{\Tr,\UG,\UnG\}$, $\ell \in \Ob_{\lambda}$, $\Os_{\ell} =   \left\{\ell \varphi | \varphi \in G_q^{P_i}\right\}$.
Then $\#\Os_{\ell} G_q= (q+1)\cdot\#\Os_{\ell}$.
\end{lem}

\begin{proof}
Let $G_i^j=  \{ \varphi \in G_q | P_i \varphi = P_j \} $. The sets $G_i^j, j \in \F_q^+$ form a partition of $G_q$. In fact, let $\varphi \in G_q$. As $G_q$  is the stabilizer group of $\C$, $P_i \varphi = P_{\overline{j}} \in \C$, so $\varphi \in G_i^{\overline{j}}$. On the other hand, if $\varphi \in G_i^j \cap G_i^k$, then $P_j = P_i \varphi = P_k$, so $j = k$.
If $\Psi \in G_i^j$, then  $\Os_{\ell}  \Psi = \Os_{\ell} G_i^j$. In fact, by Lemma \ref{lem5:eq_im}, if $\Psi' \in G_i^j$ then $\Os_{\ell}  \Psi' = \Os_{\ell}  \Psi$. Finally, consider  $\Psi_j \in G_i^j$, $j \in \F_q^+$.
Then $\Os_{\ell} G_q = \bigcup\limits_{j \in \F_q^+}\Os_{\ell}  G_i^j= \bigcup\limits_{j \in \F_q^+}\Os_{\ell} \Psi_j$.
The sets $\Os_{\ell} \Psi_j$, $j \in \F_q^+$, are disjoint.
In fact, a line $\ell' \in \Os_{\ell} \Psi_m \cap  \Os_{\ell} \Psi_n , m \neq n$, would be a line of type $\lambda$, i.e. a unisecant of $\C$, passing through the distinct points $P_m, P_n \in \C$.  Moreover, as $\Psi_j$ is a bijection,    $  \#\Os_{\ell}\Psi_j = \#\Os_{\ell}$.
Therefore,  $\#\Os_{\ell} G_q =  \sum\limits_{j \in \F_q^+}\#\Os_{\ell} \Psi_j= \sum\limits_{j \in \F_q^+}\#\Os_{\ell}= (q+1)\cdot \#\Os_{\ell}$.
\end{proof}

\begin{lem}\label{lem5:allcovered}
Let $\lambda \in \{\UG,\UnG\}$, $\ell \in \O_\lambda$. Let $P_i$ be a point of $\C$. Then there exists a line $\overline{\ell} \in \Ob_{\lambda_i}$ such that
$\ell \in \Os_{\overline{\ell}} G_q$, where  $\Os_{\overline{\ell}} =  \left\{\overline{\ell} \varphi | \varphi \in G_q^{P_i}\right\}$.
\end{lem}
\begin{proof}
The line   $\ell$ is a unisecant, so there exists $P_j$ such that $P_j \in \ell$. As $G_q$ acts transitively on $\C$, there exists $\Psi \in G_q$ such that $P_j \Psi = P_i$. Let $\overline{\ell} = \ell \Psi$. Then $\overline{\ell}$ is of the same type  $\lambda$ as $\ell$, i.e. $\overline{\ell}$ is a unisecant, and  $P_j \in \ell$ implies $P_j \Psi = P_i \in \ell \Psi= \overline{\ell}$, i.e.  $\overline{\ell} \in \Ob_{\lambda_i}$. Finally, $\ell = \overline{\ell}\Psi^{-1}$ implies $\ell \in \Os_{\overline{\ell}} G_q$.
\end{proof}

\begin{lem}\label{lem5:partition}
Let $\lambda \in \{\UG,\UnG\}$, let  $P_i \in \C$, $\ell^1, \dots, \ell^m \in \Ob_{\lambda_i}$, $\Os_{\ell^j} =   \left\{\ell^j \varphi | \varphi \in G_q^{P_i}\right\}$, $j \in  1, \dots, m$ .
If $\{ \Os_{\ell^1}, \dots,  \Os_{\ell^m}\}$ is a partition of  $\Ob_{\lambda_i}$, then  $\{ \Os_{\ell^1} G_q, \dots,  \Os_{\ell^m} G_q\}$ is a partition of  $\O_{\lambda}$.
\end{lem}
\begin{proof}
Let $\overline{\ell} \in \O_\lambda$. By Lemma \ref{lem5:allcovered}, there exists
$\ell'  \in \Ob_{\lambda_i}$ such that $\overline{\ell} \in \Os_{\ell'} G_q$,  $\Os_{\ell'} =   \left\{\ell' \varphi | \varphi \in G_q^{P_i}\right\}$. By hypothesis, there exists $\ell^{\overline{j}}, \overline{j} \in \{1, \dots, m\}$, such that  $\Os_{\ell'} = \Os_{\ell^{\overline{j}}}$.
By Lemma \ref{lem5:im_diff},  $\Os_{\ell^j} \neq \Os_{\ell^{k}}, j \neq k,$ implies
$\Os_{\ell^j}G_q \neq \Os_{\ell^{k}}G_q$.
\end{proof}

In the rest of the section we denote by $c,d$ or  $c_i, d_i$ the elements of the matrix of the form \eqref{eq5_M_P0} corresponding to a projectivity $\varphi \in G_q^{P_0}$ or $\varphi_i \in G_q^{P_0}$, respectively. Also, given a  $3 \times 4$ matrix $\D$,  we denote by  $det_i(\D)$ the determinant of the $3 \times 3$ matrix obtained deleting the $i$-th column of $\D$.

\begin{thm}\label{th5:O4UGorbit}
For any $q\ge5$, in $\PG(3,q)$, for the twisted cubic $\C$ of \eqref{eq2_cubic},
 the non-tangent unisecants in a $\Gamma$-plane \emph{(}i.e. $\UG$-lines, class $\O_4=\O_\UG$\emph{)} form an orbit under~$G_q$ if $q$ is odd and two orbits of size $q+1$ and $q^2-1$ if $q$ is even. Moreover, for $q$ even, the orbit of size $q+1$  consists of the lines in the regulus complementary to that of the tangents. Also, for $q$ even, the $(q+1)$-orbit and $(q^2-1)$-orbit can be represented in the form $\{\ell_1\varphi|\varphi\in G_q\}$ and $\{\ell_2\varphi|\varphi\in G_q\}$, respectively, where $\ell_j$ is a line such that  $\ell_1=\overline{P_0\Pf(0,1,0,0)}$, $\ell_2=\overline{P_0\Pf(0,1,1,0)}$, $P_0=\Pf(0,0,0,1)\in\C$.
\end{thm}

\begin{proof}
Let $\Ob_{\UG_0} = \{\ell \in \O_\UG | P_0 \in \ell\}$ be the set of $\UG$-lines through $P_0$.
By Lemma \ref{lem5:im_cont}, $\Ob_{\UG_0} G_q^{P_0}  = \Ob_{\UG_0}$, so we can  consider the orbits of  $\Ob_{\UG_0} $ under the group $G_q^{P_0}$.
In $\pi_\t{osc}(0)$, there are $q+1$ unisecants through $P_0$, one of which is a tangent whereas the other $q$ are $\UG$-lines; so $\#\Ob_{\UG_0}=q$. By \eqref{eq2_plane}, \eqref{eq2_osc_plane}, $\pi_\t{osc}(0)$ has equation $x_0 = 0$. By Lemma \ref{eqTangents}, the tangent $\T_0$ to $\C$ at $P_0$ has equation $x_0 = x_1 = 0$.

Let $P' = \Pf(0,1,0,0)$, $\ell'= \overline{P' P_0}$. By above, $P' \in\pi_\t{osc}(0)$, $P'\not \in\T_0$, whence $\ell'\in\Ob_{\UG_0}$. Let $\Os_{\ell'} =   \left\{\ell' \varphi | \varphi \in G_q^{P_0}\right\}$. If $\varphi\in  G_q^{P_0}$, then , by \eqref{eq5_M_P0}, $P'\varphi=\Pf (0, d, 2cd,3c^2d)= \Pf (0, 1, 2c,  3c^2)\not \in\T_0$. So,  $\ell' \varphi$ is of type $\UG$ and $P_0 \in \ell'$ implies $P_0\varphi = P_0 \in \ell' \varphi$ , whence $\Os_{\ell'}\subseteq\Ob_{\UG_0}$.

Now we determine $\#\Os_{\ell'}$. Let  $\varphi_1, \varphi_2 \in  G_q^{P_0}, \varphi_1 \neq \varphi_2$,  $Q'= P'  \varphi_1$,  $R'= P'  \varphi_2$.
By~\eqref{eq5_M_P0} with $d_1, d_2 \ne 0$, we have
\begin{align*}
 &Q'= \Pf (0, d_1, 2c_1d_1,  3c_1^2d_1)= \Pf (0, 1, 2c_1,  3c_1^2), \db\\
&  R'= \Pf (0, d_2, 2c_2d_2,  3c_2^2d_2) = \Pf (0, 1, 2c_2,  3c_2^2).
\end{align*}
Obviously, $\ell' \varphi_1 \neq \ell' \varphi_2$ if and only if  $P_0, Q', R'$ are not collinear, i.e. the matrix $\D'= [ P_0, Q', R']^{tr}$ has the maximum rank. We obtain
\begin{align*}
det_1(\D') = 2c_2-2c_1, ~det_2(\D')= det_3(\D')= det_4(\D') = 0 .
\end{align*}

If $q$ is odd, fixed $d \neq 0$ in \eqref{eq5_M_P0}, and varying $c \in \F_q$, we obtain $q$ different images of $\ell'$, i.e.  $\Ob_{\UG_0} = \Os_{\ell'}$.
Then, by Lemma \ref{lem5:partition}, $\Os_{\ell'} G_q = \O_\UG$.

Let $q$ be even.

We have $det_1(\D') = 0$, so  $\Os_{\ell'} =  \left\{ \ell' \right \}$.

Consider $P'' = \Pf(0,1,1,0)\not\in\T_0$ and $\ell''= \overline{P'' P_0}$.  As $\ell'$ has equation $x_0 =  x_2 = 0$, we have $P'' \notin \ell'$; so $\ell'' \neq \ell'$, i.e. $\ell'' \notin \Os_{\ell'}$.  Let $\Os_{\ell''} =   \left\{\ell'' \varphi | \varphi \in G_q^{P_0}\right\}$.  If $\varphi\in  G_q^{P_0}$, then $\ell'' \varphi$ is of type $\UG$ and $P_0 \in \ell''$ implies $P_0\varphi = P_0 \in \ell'' \varphi$ , whence $\Os_{\ell''}\subseteq\Ob_{\UG_0}$.
 Let  $\varphi_1, \varphi_2 \in  G_q^{P_0}, ~\varphi_1 \neq \varphi_2$,  $Q''= P  \varphi_1$,  $R''= P  \varphi_2$.
By~\eqref{eq5_M_P0} with $d_1, d_2 \ne 0$, we have
\begin{align*}
 Q''=  \Pf (0, 1, d_1,  c_1^2+c_1 d_1 ), ~R''= \Pf (0, 1, d_2,  c_2^2+c_2 d_ 2).
\end{align*}
As above, $\ell'' \varphi_1 \neq \ell'' \varphi_2$ if and only if  $P_0, Q'', R''$ are not collinear, i.e. the matrix $\D''= [ P_0, Q'', R'']^{tr}$ has the maximum rank. We obtain
\begin{align*}
det_1(\D'') = d_2-d_1,  ~det_2(\D'')= det_3(\D'')= det_4(\D'') = 0 .
\end{align*}
Fixed $c$ and varying $d \in \F^*$, we obtain $q-1$ different images of $\ell''$, i.e.  $\#\Os_{\ell''}=q-1$.

As $\Os_{\ell'} \cap \Os_{\ell''}= \emptyset$ and $\#\Ob_{\UG_0}=q$, $\{ \Os_{\ell'} ,\Os_{\ell''}\}$ is a partition of  $\Ob_{\UG_0}$. Then,  by Lemma \ref{lem5:partition}, $\{ \Os_{\ell'}G_q ,\Os_{\ell''}G_q \}$ is a partition of $\O_\UG$.
By Lemma \ref{lem5:cardinality}, $\#\Os_{\ell'} G_q= q+1$, $\#\Os_{\ell''} G_q= (q-1)(q+1)$.

Finally, on content of the $(q+1)$-orbit $\Os_{\ell'} G_q$ see Theorem \ref{th2_Hirs}(iv)(a).
\end{proof}

\begin{thm}\label{th5:O5UnGorbit}
Let $q\ge5$. In $\PG(3,q)$, for the twisted cubic $\C$ of \eqref{eq2_cubic},
 the non-tangent unisecants not in a $\Gamma$-plane \emph{(}i.e. $\UnG$-lines, class $\O_5=\O_\UnG$\emph{)} form an orbit under~$G_q$ if $q$ is even and two orbits of size $\frac{1}{2}(q^3-q)$  if $q$ is odd. Moreover, for $q$ odd, the two orbits can be represented in the form $\{\ell_j\varphi|\varphi\in G_q\}$, $j=1,2$, where $\ell_j$ is a line such that $\ell_1=\overline{P_0\Pf(1,0,1,0)}$,  $\ell_2=\overline{P_0\Pf(1,0,\rho,0)}$, $P_0=\Pf(0,0,0,1)\in\C$, $\rho$ is not a square.
\end{thm}

\begin{proof}
We act similarly to the proof of Theorem \ref{th5:O4UGorbit}.
Let $\Ob_{\UnG_0} = \{\ell \in \O_\UnG | P_0 \in \ell\}$.
By Lemma \ref{lem5:im_cont}, $\Ob_{\UnG_0} G_q^{P_0}  = \Ob_{\UnG_0}$, so we can  consider the orbits of  $\Ob_{\UnG_0} $ under~$G_q^{P_0}$.
In total, through $P_0$ there are $q^2+q+1$ lines, $q+1$ of which are unisecants in $\pi_\t{osc}(0)$, other $q$ are real chords, and the remaining $q^2-q$ are $\UnG$-lines. So, $\#\Ob_{\UnG_0}=q^2-q$. The equation of $\pi_\t{osc}(0)$ is $x_0 = 0$. The tangent $\T_0$ to $\C$ in $P_0$ has equation $x_0 = x_1 = 0$.

Let $P' = \Pf(1,0,1,0)$ and $\ell'= \overline{P' P_0}\not\in\pi_\t{osc}(0)$.
 Also,  $\ell'$ is not a real chord, as $\ell'$ has equation $x_0 = x_2 , x_1 = 0$ and $ \C \cap \ell'= P_0$. Thus, $\ell'$ is a $\UnG$-line.
Let $\Os_{\ell'} =   \left\{\ell' \varphi | \varphi \in G_q^{P_0}\right\}$. We have $\Os_{\ell'}\subseteq\Ob_{\UnG_0}$, as $\ell' \varphi$ is a  $\UnG$-line and $P_0 \in \ell'$ implies $P_0 \varphi= P_0 \in \ell'\varphi$.

We  find $\#\Os_{\ell'}$. Let  $\varphi_1, \varphi_2 \in  G_q^{P_0}, \varphi_1 \neq \varphi_2$,  $Q'= P'  \varphi_1$,  $R'= P'  \varphi_2$.
By~\eqref{eq5_M_P0},
\begin{align*}
 Q'= \Pf (1, c_1, c_1^2+d_1^2, c_1^3+ 3c_1d_1^2),~R'= \Pf (1,  c_2, c_2^2+d_2^2, c_2^3+ 3c_2d_2^2).
\end{align*}
Obviously, $\ell' \varphi_1 \neq \ell' \varphi_2$ if and only if  $P_0, Q', R'$ are not collinear, i.e.  the matrix $\D'= [ P_0, Q', R']^{tr}$ has the maximum rank. We obtain
\begin{align*}
 & det_1(\D') = c_1(c_2^2+d_2^2) - c_2(c_1^2+d_1^2), ~det_2(\D') = c_2^2+d_2^2-(c_1^2+d_1^2),\\
 & det_3 (\D')= c_2-c_1, ~det_4(\D') = 0.
\end{align*}
If $c_2 \neq c_1$, then $det_3 (\D') \neq 0.$

If $q$ is even and $c_2 = c_1$, then $det_2(\D') = d_2^2-d_1^2 = (d_2-d_1)^2$, so $det_2 (\D') = 0$ if and only if $d_2 = d_1$. Therefore,  $\varphi_1 \neq \varphi_2$ implies $\ell' \varphi_1 \neq \ell' \varphi_2$. It means that $\Os_{\ell'}=\Ob_{\UnG_0}$ and
$\#\Os_{\ell'} = \# G_q^{P_0}=q(q-1)$, see Corollary \ref{cor5:size}.
Then,  by Lemma~\ref{lem5:partition}, $ \Os_{\ell'}G_q = \O_\UG$ and by Lemma \ref{lem5:cardinality}, $\#\Os_{\ell'} G_q= q(q-1)(q+1)$.

Let $q$ be odd.

If $c_2 = c_1$, then $det_2(\D') = (d_2-d_1)(d_2+d_1)$, so $det_2(\D') = 0$ if $d_1=-d_2$. In this case also $det_1(\D') = 0$. Therefore, given $\varphi_1 \in G_q^{P_0}$, if and only if we take $\varphi_2 \in G_q^{P_0}$ with $c_2 = c_1, d_2=-d_1$, then   $\varphi_1 \neq \varphi_2$ and $\ell' \varphi_1 = \ell' \varphi_2$. It means that $\#\Os_{\ell'} = \frac{1}{2}q(q-1)$, see \eqref{eq5_M_P0}.

Consider $P'' = \Pf(1,0,\rho,0)$, $\rho$ is not a square, and $\ell''= \overline{P'' P_0}\not\in\pi_\t{osc}(0)$. Also,  $\ell''$ is not a real chord, as $\ell''$ has equation $\rho x_0 = x_2 , x_1 = 0$ and $ \C \cap \ell''=P_0$. Thus, $\ell''$ is a $\UnG$-line. Also $\ell'' \notin \Os_{\ell'}$. In fact, if  $\ell'' \in \Os_{\ell'}$, then $\varphi \in G_q^{P_0}$ such that $P_0, P'\varphi,P''$ are collinear would exist. It means that the matrix  $\D_\varphi= [ P_0, P'\varphi, P'']^{tr}$ should have rank 2.
As $P'\varphi = \Pf (0, c, c^2+d^2, c^3+ 3cd^2)$, we have
\begin{align*}
 det_1(\D_\varphi) = -\rho c, ~det_2(\D_\varphi) = c^2+d^2-\rho, ~det_3(\D_\varphi) = c,~
det_4(\D_\varphi) = 0.
\end{align*}
Thus, $det_3(\D_\varphi) = 0$ implies $c = 0$. Then $det_2(\D_\varphi) = d^2-\rho$, that cannot be equal to $0$ as $\rho$ is not a square; contradiction.

Let $\Os_{\ell''} =   \left\{\ell'' \varphi | \varphi \in G_q^{P_0}\right\}$.  Let  $\varphi_1, \varphi_2 \in  G_q^{P_0}, \varphi_1 \neq \varphi_2$,  $Q''= P''  \varphi_1$,  $R''= P''  \varphi_2$.
By \eqref{eq5_M_P0},
\begin{align*}
 Q''=  \Pf (1, c_1, c_1^2 + \rho d_1^2,  c_1^3+3 \rho c_1 d_1^2 ),~ R''= \Pf (1, c_2, c_2^2 + \rho d_2^2,  c_2^3+3 \rho c_2 d_2^2 ).
\end{align*}
Obviously, $\ell'' \varphi_1 \neq \ell'' \varphi_2$ if and only if  $P_0, Q'', R''$ are not collinear, i.e. the matrix $\D''= [ P_0, Q'', R'']^{tr}$ has the maximum rank. We have
\begin{align*}
&det_1(\D'') = c_1(c_2^2+ \rho d_2^2) - c_2 (c_1^2+ \rho d_1^2),~
det_2(\D'') = c_2^2+ \rho d_2^2-(c_1^2+ \rho d_1^2),\db\\
&det_3(\D'') = c_2-c_1,~det_4(\D'') = 0.
\end{align*}
If $c_2 \neq c_1$, then $det_3(\D'') \neq 0$.
If $c_2 = c_1$ then $det_2(\D'') = \rho (d_2-d_1)(d_2+d_1)$, so $det_2(\D'') = 0$ if $d_1=-d_2$. In this case also $det_1(\D'') = 0$. Therefore, given $\varphi_1 \in G_q^{P_0}$ if and only if we take $\varphi_2 \in G_q^{P_0}$ with $c_2 = c_1, d_2=-d_1$ then we obtain $\varphi_1 \neq \varphi_2$ and $\ell'' \varphi_1 = \ell' \varphi_2$. It means that $\#\Os_{\ell''} = \frac{1}{2}q(q-1)$, see \eqref{eq5_M_P0}.

As $\Os_{\ell'} \cap \Os_{\ell''}= \emptyset$ and $\#\Ob_{\UnG_0}=q(q-1)$, $\{ \Os_{\ell'} ,\Os_{\ell''}\}$ is a partition of  $\Ob_{\UnG_0}$. Then,  by Lemma \ref{lem5:partition},$\{ \Os_{\ell'}G_q ,\Os_{\ell''}G_q \}$ is a partition of $\O_\UnG$.
By Lemma \ref{lem5:cardinality}, $\#\Os_{\ell'} G_q=\#\Os_{\ell''} G_q= \frac{1}{2}q(q-1)(q+1)$.
\end{proof}

\begin{corollary}\label{cor5:O5'EG}
Let $q\not\equiv0\pmod 3$. In $\PG(3,q)$, for the twisted cubic $\C$ of \eqref{eq2_cubic},
    the external lines in a $\Gamma$-plane   \emph{(}class $\O_5'=\O_\EG$\emph{)} form an orbit under~$G_q$ if $q$ is even and two orbits of size $(q^3-q)/2$  if $q$ is odd. Moreover, for $q$ odd, the two orbits can be represented in the form $\{\ell_j\varphi|\varphi\in G_q\}$, $j=1,2$, where $\ell_j=\pk_0\cap\pk_j$ is the intersection line of planes $\pk_0$ and $\pk_j$ such that
           $\pk_0=\boldsymbol{\pi}(1,0,0,0)=\pi_\t{\emph{osc}}(0)$, $\pk_1=\boldsymbol{\pi}(0,-3,0,-1)$,  $\pk_2=\boldsymbol{\pi}(0,-3\rho,0,-1)$, $\rho$ is not a square, cf. \eqref{eq2_plane}, \eqref{eq2_osc_plane}.
\end{corollary}

\begin{proof}
The assertion follows from Theorems \ref{th2_Hirs}(iv)(a), \ref{th5_null_pol}, and \ref{th5:O5UnGorbit}.
  The null polarity $\A$ \eqref{eq2_null_pol}
  maps the points $P_0=\Pf(0,0,0,1)$, $P'=\Pf(1,0,1,0)$, and $P''=\Pf(1,0,\rho,0)$
of Proof of Theorem \ref{th5:O5UnGorbit}
to the planes $\pk_0=\boldsymbol{\pi}(1,0,0,0)$, $\pk_1=\boldsymbol{\pi}(0,-3,0,-1)$, and $\pk_2=\boldsymbol{\pi}(0,-3\rho,0,-1)$, respectively. The $\UnG$-lines $\ell'=\overline{P_0P'}$ and  $\ell''=\overline{P_0P''}$ are mapped to $\EA$-lines so that $\ell'\A=\pk_0\cap\pk_1\triangleq\ell_1$ and $\ell''\A=\pk_0\cap\pk_2\triangleq\ell_2$.
\end{proof}

\section{Orbits under $G_q$ of external lines with respect to the cubic $\C$ meeting the axis of the pencil of osculating planes, $q\equiv 0 \pmod 3$ (orbits of $\EA$-lines)}\label{sec:orbEA}

In the following we consider $q\equiv 0 \pmod 3$, $q \geq 9$, and denote by $\ell_{\Ar}$  the  axis of $\Gamma$  and by  $P_{\Ar}$ the point  $\Pf (0, 1, 0,0)$. The line $\ell_{\Ar}$ is the intersection of the osculating planes, so has equation $x_0 =  x_3 = 0$, and $P_{\Ar} \in \ell_{\Ar}$. Recall that by Theorem \ref{th2_Hirs}(iv)(b),  $\ell_{\Ar}$ is fixed by  $G_q$.

\begin{notation}\label{notation_4}
In addition to Notations \ref{notation_1} and \ref{notation_3}, the following notation is used.
\begin{align*}
&P_t^{\Ar} && \t{the point } \Pf(0,1,t,0) \t{ of } \ell_{\Ar} \t{ with } t\in\F_q;\db  \\
&P_{\infty}^{\Ar} && \t{the point } \Pf(0,0,1,0) \t{ of } \ell_{\Ar};\db  \\
&G_q^{P_t^{\Ar}}&&  \t{the subgroup of } G_q \t{ fixing }P_t^{\Ar} \t{ with } t\in\F_q^+;\db  \\
&\Ob_{\EA_i} &&\t{the set of lines from $\O_\EA$ through $P_i^\Ar$,}\db\\
&&&\t{i.e. }\Ob_{\EA_i}\triangleq \{\ell \in \O_\EA | P_i^{\Ar} \in \ell\}.
\end{align*}
\end{notation}

\begin{lem}\label{lem5:GqtranslA}
\looseness - 1    Let $q\equiv0\pmod 3$, $q \geq 9$. The group $G_q$ acts transitively on $\ell_{\Ar}$.
\end{lem}

\begin{proof}
If we take $\varphi \in G_q$ whose matrix in the form \eqref{eq2_M} has $a=0$, $b=c=d=1$, then $P_0^{\Ar} \varphi =\Pf (0, 0, 1,0)=P_{\infty}^{\Ar}$.
If we take $\varphi \in G_q$ whose matrix in the form \eqref{eq2_M} has $a=d=1, b=0, c=-n$, then $ P_0^{\Ar} \varphi =\Pf (0, 1, n,0) = P_n^{\Ar}$.
\end{proof}

\begin{lem}
The general form of the matrix $\MM$ corresponding to a projectivity of $G_q^{P_0^{\Ar}}$  is as follows:
\begin{align}\label{eq_M_PA}
  \mathbf{M}=\left[
 \begin{array}{cccc}
 1&0&0&0\\
 0&d&0&0\\
 0&-bd&d^2&0\\
 b^3&b^2d&bd^2&d^3
 \end{array}
  \right],~b \in \F_q, d \in \F_q^*.\db
\end{align}
\end{lem}

\begin{proof} Let $\MM$ be the matrix corresponding to a projectivity $\Psi \in G_q$;
 by \eqref{eq2_M}, $ [0,1,0,0]\times\mathbf{M} = [0,a^2d-abc,bc^2-acd,0]$.
Then  $\Psi \in G_q^{P_0^{\Ar}}$ if and only if
\begin{align}\label{eq5_a}
&bc^2-acd = 0,~
 a^2d-abc \neq 0.
\end{align}

If $a = 0$, then $bc^2=0$ that implies $det(\MM) = 0$, contradiction, so we can fix $a = 1$.  Then the 1-st equality of \eqref{eq5_a} becomes $c(bc-d) = 0$.
If  $a = 1$ and $bc-d = 0$, also  $a^2d-abc = 0$. Therefore, $c = 0, d \neq 0$.
\end{proof}

Lemma \ref{lem5:GqtranslA} allows to prove the following lemmas and corollary in analogous way to Lemmas \ref{lem5:conj}, \ref{lem5:im_cont}--\ref{lem5:partition}, and Corollary \ref{cor5:size}.

\begin{lem}\label{lem5:conjA}
$G_q^{P_i^{\Ar}}$ and  $G_q^{P_j^{\Ar}}$ are conjugate subgroups of $G_q$.
\end{lem}
\begin{proof}
By Lemma \ref{lem5:GqtranslA}, $G_q$ acts transitively on $\ell_{\Ar}$, so there exists  $\Psi \in G_q$ such that $P_i^{\Ar} \Psi = P_j^{\Ar}$. Then   $ \Psi^{-1} G_q^{P_i^{\Ar}} \Psi =G_q^{P_j^{\Ar}}$. In fact, let $\varphi \in G_q^{P_i^{\Ar}}$. Then  $P_j^{\Ar} \Psi^{-1} \varphi \Psi = P_i^{\Ar} \varphi \Psi = P_i^{\Ar} \Psi = P_j^{\Ar}$.
On the other hand, let $\gamma \in G_q^{P_j^{\Ar}}$. Then  $P_i^{\Ar} \Psi \gamma \Psi^{-1} = P_j^{\Ar}  \gamma \Psi^{-1} = P_j^{\Ar}  \Psi^{-1}= P_i^{\Ar}$. It means that $\Psi \gamma \Psi^{-1} \in G_q^{P_i^{\Ar}}$, i.e. $\gamma \in \Psi^{-1} G_q^{P_i^{\Ar}} \Psi$.
\end{proof}

\begin{corollary}\label{cor5:sizeA}
For all $t\in\F_q^+$, we have $\#G_q^{P_t^{\Ar}}= q(q-1)$.
\end{corollary}

\begin{proof}
By \eqref{eq_M_PA}, $\#G_q^{P_0^{\Ar}} = q(q-1)$. By Lemma \ref{lem5:conjA}, there exists  $\Psi \in G_q$ such that  $ \Psi^{-1}G_q^{P_0^{\Ar}} \Psi =G_q^{P_t^{\Ar}}$,  $t\in\F_q^+$. Then $G_q^{P_0^{\Ar}} \Psi =  \Psi  G_q^{P_t^{\Ar}}$. As a finite group and its cosets have the same cardinality, $\#G_q^{P_0^{\Ar}}=\#G_q^{P_0^{\Ar}}\Psi=\#\Psi G_q^{P_t^{\Ar}}=\# G_q^{P_t^{\Ar}}$.
\end{proof}

\begin{lem}\label{lem5:im_contA}
We have $\Ob_{\EA_i} G_q^{P_i^{\Ar}}  = \Ob_{\EA_i}$.
\end{lem}
\begin{proof}
Let $\ell \in \Ob_{\EA_i}, \varphi \in G_q^{P_i^{\Ar}}$.  Then $P_i^{\Ar} \in \ell$, so $P_i^{\Ar} \varphi =P_i^{\Ar} \in \ell \varphi$. As $\ell$ of type $\EA$ implies $\ell \varphi$ of type $\EA$,   $\ell \varphi \in  \Ob_{\EA_i}$.
 On the other hand, if $I$ is the identity element of $G_q^{P_i}$, $\Ob_{\lambda_i} G_q^{P_i} \supseteq \Ob_{\lambda_i} I = \Ob_{\lambda_i}$.
\end{proof}

\begin{lem}\label{lem5:eq_imA}
Let $\ell \in \Ob_{\EA_i}$, $\Os_{\ell} =   \left\{\ell \varphi | \varphi \in G_q^{P_i^{\Ar}}\right\}$, $\Psi_1, \Psi_2 \in G_q$.
If  $P_i^{\Ar} \Psi_1 = P_i^{\Ar} \Psi_2 = P_j^{\Ar}$ then $\Os_{\ell} \Psi_1  = \Os_{\ell} \Psi_2$.
\end{lem}
\begin{proof}
As $  P_i^{\Ar} \Psi_1 \Psi_2^{-1} = P_j^{\Ar} \Psi_2^{-1}=  P_i^{\Ar}$, $\Psi_1 \Psi_2^{-1} \in G_q^{P_i^{\Ar}}$.
Let $\overline{\ell} \in \Os_{\ell}\Psi_1$. Then $\overline{\ell} =  \ell \varphi \Psi_1,  \varphi \in G_q^{P_i^{\Ar}}$. This implies $ \overline{\ell}  \Psi_2^{-1} =  \ell \varphi \Psi_1  \Psi_2^{-1} \in \Os_{\ell}$, whence   $\overline{\ell}  \in \Os_{\ell}\Psi_2$. The proof of the other inclusion is analogous.
\end{proof}

\begin{lem}\label{lem5:im_diffA}
Let  $\ell_1, \ell_2, \in \Ob_{\EA_i},~  \Os_{\ell_1} =   \left\{\ell_1 \varphi | \varphi \in G_q^{P_i^{\Ar}}\right\}$,\\  $\Os_{\ell_2} =   \left\{\ell_2 \varphi | \varphi \in G_q^{P_i^{\Ar}}\right\}$. If  $\Os_{\ell_1} \cap \Os_{\ell_2} = \emptyset$ then $\Os_{\ell_1} G_q \cap \Os_{\ell_2} G_q = \emptyset$.
\end{lem}
\begin{proof}
Suppose  $\overline{\ell}  \in \Os_{\ell_1} G_q \cap \Os_{\ell_2} G_q$. Then also  $\overline{\ell}$ is a line of type $\EA$; let $P_j^{\Ar} = \ell_{\Ar} \cap \overline{\ell}$.
As $\overline{\ell} \in \Os_{\ell_1} G_q$, $\overline{\ell} = \ell_1 \varphi_1 \Psi_1$,   $\varphi_1 \in G_q^{P_i^{\Ar}}$, $\Psi_1 \in G_q$. Let $\ell'= \ell_1 \varphi_1$. As $P_i^{\Ar} \in \ell_1$, $P_i^{\Ar}\varphi_1 = P_i^{\Ar} \in \ell_1\varphi_1 = \ell'$.  As $P_j^{\Ar} \in \overline{\ell} = \ell' \Psi_1$ and $\ell_1, \ell', \overline{\ell} $ are of type $\EA$ and $\Psi_1 \in  G_q $, $P_i^{\Ar} \Psi_1$ is the only point of $ \overline{\ell}$ belonging to $\ell_{\Ar}$, i.e. $P_i^{\Ar} \Psi_1 = P_j^{\Ar}$. Analogously, $\overline{\ell} \in \Os_{\ell_2} G_q$ implies $\overline{\ell} = \ell_2 \varphi_2 \Psi_2$,   $\varphi_2 \in G_q^{P_i^{\Ar}}$,  $P_i^{\Ar} \Psi_2 =P_j^{\Ar}$.
Then $P_i^{\Ar} \Psi_1 \Psi_2^{-1} =P_j^{\Ar} \Psi_2^{-1}=P_i^{\Ar}$ that implies $\Psi_1 \Psi_2^{-1} \in  G_q^{P_i^{\Ar}}$. Finally, $\ell_1 \varphi_1 \Psi_1 = \ell_2 \varphi_2 \Psi_2$ implies $\ell_1 \varphi_1 \Psi_1 \Psi_2^{-1} \varphi_2^{-1} = \ell_2$, whence  $\ell_2 \in  \Os_{\ell_1}$.
\end{proof}

\begin{lem}\label{lem5:cardinalityA}
Let $\ell \in  \Ob_{\EA_i}$, $\Os_{\ell} =   \left\{\ell \varphi | \varphi \in G_q^{P_i^{\Ar}}\right\}$.
Then $\#\Os_{\ell} G_q= (q+1)\cdot\#\Os_{\ell}$.
\end{lem}
\begin{proof}
Let $G_i^j=  \{ \varphi \in G_q | P_i^{\Ar} \varphi = P_j^{\Ar} \} $. The sets $G_i^j, j \in \F_q^+$ form a partition of $G_q$. In fact, let $\varphi \in G_q$. By Theorem \ref{th2_Hirs}(iv)(b),  the line $\ell_{\Ar}$  is fixed by  $ G_q$, so $P_i^{\Ar} \varphi = P_{\overline{j}}^{\Ar} \in \ell_{\Ar}$: it means that $\varphi \in G_i^{\overline{j}}$. On the other hand, if $\varphi \in G_i^j \cap G_i^k$, then $P_j^{\Ar} = P_i^{\Ar} \varphi = P_k^{\Ar}$, so $j = k$.
If $\Psi \in G_i^j$, then  $\Os_{\ell}  \Psi = \Os_{\ell} G_i^j$. In fact, by Lemma \ref{lem5:eq_imA}, if $\Psi' \in G_i^j$ then $\Os_{\ell}  \Psi' = \Os_{\ell}  \Psi$. Finally, consider  $\Psi_j \in G_i^j$, $j \in \F_q^+$.
Then $\Os_{\ell} G_q = \bigcup\limits_{j \in \F_q^+}\Os_{\ell}  G_i^j= \bigcup\limits_{j \in \F_q^+}\Os_{\ell} \Psi_j$.
The sets $\Os_{\ell} \Psi_j$, $j \in \F_q^+$, are disjoint.
In fact, a line $\ell \in \Os_{\ell} \Psi_m \cap  \Os_{\ell} \Psi_n , m \neq n$, would be a line of type $\EA$ passing through the distinct points $P_m^{\Ar}, P_n^{\Ar} \in \ell_{\Ar}$.  Moreover, as $\Psi_j$ is a bijection,    $  \#\Os_{\ell}\Psi_j = \#\Os_{\ell}$.
Therefore,  $\#\Os_{\ell} G_q =  \sum\limits_{j \in \F_q^+}\#\Os_{\ell} \Psi_j= \sum\limits_{j \in \F_q^+}\#\Os_{\ell}= (q+1)\cdot \#\Os_{\ell}$.
\end{proof}

\begin{lem}\label{lem5:allcoveredA}
Let $\ell \in  \Ob_{\EA}$. Let $ P_i^{\Ar}$ be a point of $\ell_{\Ar}$. Then there exists a line $\overline{\ell} \in \Ob_{\EA_i}$ such that
$\ell \in \Os_{\overline{\ell}} G_q$, where  $\Os_{\overline{\ell}} =  \left\{\overline{\ell} \varphi | \varphi \in G_q^{P_i^{\Ar}}\right\}$.
\end{lem}
\begin{proof}
As $\ell \in  \Ob_{\EA}$, there exists $P_j^{\Ar} \in \ell_{\Ar}$, such that  $P_j^{\Ar} \in \ell$. By Lemma \ref{lem5:GqtranslA}, $G_q$ acts transitively on $\ell_{\Ar}$, so there exists $\Psi \in G_q$ such that $P_j^{\Ar} \Psi = P_i^{\Ar}$. Let $\overline{\ell} = \ell \Psi$. Then $\overline{\ell}$ is of type $\EA$ and  $P_j^{\Ar} \in \ell$ implies $P_j^{\Ar} \Psi = P_i^{\Ar} \in \ell \Psi= \overline{\ell}$, i.e.  $\overline{\ell} \in \Ob_{\EA_i}$. Finally, $\ell = \overline{\ell}\Psi^{-1}$ implies $\ell \in \Os_{\overline{\ell}} G_q$.
\end{proof}

\begin{lem}\label{lem5:partitionA}
Let  $ P_i^{\Ar} \in \ell_{\Ar}$, $\ell^1, \dots, \ell^m, \in \Ob_{\EA_i}$, $\Os_{\ell^j} =   \left\{\ell^j \varphi | \varphi \in  G_q^{P_i^{\Ar}}\right\}$, $j \in  1, \dots, m$.
If $\{ \Os_{\ell^1}, \dots,  \Os_{\ell^m}\}$ is a partition of  $\Ob_{\EA_i}$, then  $\{ \Os_{\ell^1} G_q, \dots,  \Os_{\ell^m} G_q\}$ is a partition of  $\O_{\EA}$.
\end{lem}
\begin{proof}
Let $\overline{\ell} \in \O_\EA$. By Lemma \ref{lem5:allcoveredA}, there exists
$\ell'  \in \Ob_{\EA_i}$ such that $\overline{\ell} \in \Os_{\ell'} G_q$,  $\Os_{\ell'} =   \left\{\ell' \varphi | \varphi \in  G_q^{P_i^{\Ar}}\right\}$. By hypothesis, there exists $\ell^{\overline{j}}$, $\overline{j} \in \{1, \dots, m\}$, such that  $\Os_{\ell'} = \Os_{\ell^{\overline{j}}}$.
By Lemma \ref{lem5:im_diffA},  $\Os_{\ell^j} \neq \Os_{\ell^{k}}, j \neq k,$ implies
$\Os_{\ell^j}G_q \neq \Os_{\ell^{k}}G_q$.
\end{proof}

\begin{lem}\label{lem5:cardstarA}
We have $\#\Ob_{\EA_i} = q^2-1$.
\end{lem}

\begin{proof}
No real cord contains the point $P_0^{\Ar}$. In fact, the line $\overline{P_0^{\Ar} P_\infty}$ has equation $x_2=x_3=0$ and contains no point $P_t, t \in \F$.
The points $P_0^{\Ar}, P_{t_1}, P_{t_2}$, with $t_1, t_2 \in \F, t_1 \neq t_2$ are collinear if and only if the matrix $\MM_{P_0^{\Ar}, P_{t_1}, P_{t_2}}= [ P_0^{\Ar}, P_{t_1}, P_{t_2}]^{tr}$ has rank 2, but $det_1(\MM_{P_0^{\Ar}, P_{t_1}, P_{t_2}}) = t_1-t_2 \neq 0$.

In total $q^2+q+1$ lines pass through the point $P_0^{\Ar}$. One is $\ell_{\Ar}$, other $q+1$ are unisecants to $\C$. Therefore, the remaining $q^2-1$ lines are of type $\EA$.
The same holds for every point of $\ell_{\Ar}$. In fact, let $\overline{\ell}$ be a $\RC$-line through a point $P_i^{\Ar}$. Then  by Lemma \ref{lem5:GqtranslA}, as $G_q$ acts transitively on $\ell_{\Ar}$, there exists $\Psi \in G_q$ such that $P_i^{\Ar} \Psi  = P_0^{\Ar}$ and $\overline{\ell}\Psi  $ would be an  $\RC$-line through $P_0^{\Ar}$, contradiction.
\end{proof}

\begin{thm}\label{th5:EAorbit}
For any $q\equiv0\pmod 3,\; q\ge9$, in $\PG(3,q)$, for the twisted cubic $\C$ of \eqref{eq2_cubic},
 the external lines meeting the axis of $\Gamma$ \emph{(}i.e. $\EA$-lines, class $\O_8=\O_\EA$\emph{)} form three orbits under~$G_q$  of size $q^3-q, (q^2-1)/2, (q^2-1)/2$. Moreover, the $(q^3-q)$-orbit and the two $(q^2-1)/2$-orbits can be represented in the form $\{\ell_1\varphi|\varphi\in G_q\}$ and $\{\ell_j\varphi|\varphi\in G_q\}$, $j=2,3$, respectively, where $\ell_j$ are lines such that $\ell_1=\overline{P_0^\Ar\Pf(0,0,1,1)}$,  $\ell_2=\overline{P_0^\Ar\Pf(1,0,1,0)}$, $\ell_3=\overline{P_0^\Ar\Pf(1,0,\rho,0)}$, $P_0^\Ar=\Pf(0,1,0,0)$, $\rho$ is not a square.
\end{thm}

\begin{proof}
Let $\Ob_{\EA_0} = \left\{\ell \in \EA | P_0^{\Ar} \in \ell \right\}$. By Lemma \ref{lem5:im_contA}, $\Ob_{\EA_0} G_q^{P_0^{\Ar}}  = \Ob_{\EA_0}$, so we can  consider the orbits of  $\Ob_{\EA_0}$ under the group $G_q^{P_0^{\Ar}}$.
Let $P' = \Pf(0,0,1,1)$ and $\ell'= \overline{P' P_0^{\Ar}}$.
The line $\ell'$ has equation $x_0 = 0, x_2=x_3$, so  $\ell' \cap \C = \emptyset$.
Let $\Os_{\ell'} =   \left\{\ell' \varphi | \varphi \in G_q^{P_0^{\Ar}}\right\}$. We find $\#\Os_{\ell'}$. Let  $\varphi_1, \varphi_2 \in  G_q^{P_0^{\Ar}}, \varphi_1 \neq \varphi_2$,  $Q'= P'  \varphi_1$,  $R'= P'  \varphi_2$.
By \eqref{eq_M_PA} with $d_1, d_2 \neq 0$,
\begin{align*}
&Q'= \Pf (b_1^3,-b_1d_1+b_1^2d_1,d_1^2+b_1d_1^2,d_1^3),\db\\
&R'= \Pf (b_2^3,-b_2d_2+b_2^2d_2,d_2^2+b_2d_2^2,d_2^3).
\end{align*}
Obviously, $\ell' \varphi_1 \neq \ell' \varphi_2$ if and only if  $P_0^{\Ar}, Q', R'$ are not collinear, i.e. if and only if the matrix $\D'= [ P_0^{\Ar}, Q', R']^{tr}$ has maximum rank. Then
\begin{align*}
&det_1(\D') = (d_1^2+b_1d_1^2)d_2^3 -  (d_2^2+b_2d_2^2)d_1^3,~det_2(\D') = 0\db \\
&det_3 (\D')= d_1^3b_2^3 - d_2^3b_1^3 = (d_1b_2 - d_2b_1)^3,\db\\
&det_4(\D') = (d_1^2+b_1d_1^2)b_2^3-(d_2^2+b_2d_2^2)b_1^3.
\end{align*}
If $d_1b_2 - d_2b_1 \neq 0$, then $det_3 (\D') \neq 0.$ If $b_2 = d_2b_1/d_1$, then  $det_1(\D') = d_1^2d_2^2 (d_1-d_2 )$. Therefore,  $det_1(\D') =0$ if and only if $d_1=d_2$ that implies $b_1=b_2$, i.e. $\varphi_1 = \varphi_2$. Therefore,
$\#\Os_{\ell'} = \# G_q^{P_0^{\Ar}}=q(q-1)$.

Now, let  $P'' = \Pf(1,0,1,0)$, $\ell''= \overline{P_0^{\Ar}P'' }$, $P''' = \Pf(1,0,\rho,0)$, $\rho$ not a square in $\F_q$, $\ell'''= \overline{P_0^{\Ar}P'''}$. As $\ell''$ has equation $x_3=0, x_0 = x_2$, and $\ell'''$ has equation $x_3=0, \rho x_0 = x_2$, no point of $\C$ belongs to $\ell'', \ell'''$ and $\ell'', \ell''' \in \EA$. Moreover, $\ell'', \ell''' \notin \Os_{\ell'}$. In fact, let $P = \Pf(1,0,s,0)$, $s \neq 0$, $\ell= \overline{P_0^{\Ar}P}$;  if $\ell  \in \Os_{\ell'}$,  $\varphi \in G_q^{P_0^{\Ar}}$ such that $P_0^{\Ar}, P'\varphi,P$ are collinear would exist. It means that the matrix  $\D'_\varphi= [ P_0^{\Ar}, P'\varphi, P]^{tr}$ should have rank 2, but as $P'\varphi = \Pf (b^3,-bd+b^2d,d^2+bd^2,d^3)$ with $d \neq 0$,
$det_1(\D'_\varphi) = -s d^3 \neq 0$.

Let $\Os_{\ell''} =   \left\{\ell'' \varphi | \varphi \in G_q^{P_0^{\Ar}}\right\}$.
We find $\#\Os_{\ell''}$.
 Let  $\varphi_1, \varphi_2 \in  G_q^{P_0^{\Ar}}, \varphi_1 \neq \varphi_2$,  $Q''= P''  \varphi_1$,  $R''= P''  \varphi_2$.
By \eqref{eq_M_PA} with $d_1, d_2 \neq 0$,
\begin{align*}
  Q''=  \Pf (1, -b_1d_1, d_1^2 ,0),~R''= \Pf (1, -b_2d_2, d_2^2 ,0).
\end{align*}
Obviously, $\ell'' \varphi_1 \neq \ell'' \varphi_2$ if and only if  $P_0^{\Ar}, Q'', R''$ are not collinear, i.e.\ if and only if the matrix $\D''= [ P_0^{\Ar}, Q'', R'']^{tr}$ has maximum rank. We have
\begin{align*}
&det_1(\D'') = det_2(\D'') = det_3(\D'') = 0,~ det_4(\D'') =  (d_1+d_2)(d_1-d_2).
 \end{align*}
If $d_1=d_2$, $det_4(\D'') =0 \; \forall b$; if $d_1=-d_2$, $det_4(\D'') =0 \; \forall b$. It means that $\#\Os_{\ell''}=  \frac{1}{2}(q-1)$.
It holds  $\ell''' \notin \Os_{\ell''}$. In fact, if $\ell'''  \in \Os_{\ell''}$,  $\varphi \in G_q^{P_0^{\Ar}}$ such that $P_0^{\Ar}, P''\varphi,P'''$ are collinear would exist. It means that the matrix  $\D''_\varphi= [ P_0^{\Ar}, P''\varphi, P']^{tr}$ should have rank 2, but as $P''\varphi = \Pf (1, -bd, d^2 ,0)$, we have
$det_4(\D''_\varphi) = d^2-\rho \neq 0$ as $\rho$ is not a square.

Let $\Os_{\ell'''} =   \left\{\ell''' \varphi | \varphi \in G_q^{P_0^{\Ar}}\right\}$.
We find $\#\O_{\ell'''}$.
 Let  $\varphi_1, \varphi_2 \in  G_q^{P_0^{\Ar}}, \varphi_1 \neq \varphi_2$,  $Q'''= P'''  \varphi_1$,  $R'''= P'''  \varphi_2$.
By \eqref{eq_M_PA} with $d_1, d_2 \neq 0$, we have
\begin{align*}
 Q'''=  \Pf (1, -\rho b_1d_1, \rho d_1^2 ,0),~R'''= \Pf (1, -\rho b_2d_2, \rho d_2^2 ,0).
\end{align*}
Obviously, $\ell''' \varphi_1 \neq \ell''' \varphi_2$ if and only if  $P_0^{\Ar}, Q''', R'''$ are not collinear, i.e. if and only if the matrix $\D'''= [ P_0^{\Ar}, Q''', R''']^{tr}$ has maximum rank. We have
\begin{align*}
 det_1(\D''') = det_2(\D''') = det_3(\D''') = 0,~ det_4(\D''') = \rho (d_1+d_2)(d_1-d_2).
\end{align*}
If $d_1=d_2$, $det_4(\D''') =0 \; \forall b$; if $d_1=-d_2$, $det_4(\D''') =0 \; \forall b$. It means that $\#\O_{\ell'''}=  \frac{1}{2}(q-1)$.
As $\Os_{\ell'}, \Os_{\ell''}, \Os_{\ell'''}$ are pairwise disjoint and by Lemma \ref{lem5:cardstarA} $\#\Ob_{\EA_0}=q(q-1)$, $\{ \Os_{\ell'} ,\Os_{\ell''}, \Os_{\ell'''}\}$ is a partition of  $\Ob_{\EA_0}$. Then,  by Lemma \ref{lem5:partitionA}, $\{ \Os_{\ell'}G_q ,\Os_{\ell''}G_q,$ $\Os_{\ell'''}G_q  \}$ is a partition of $\O_\EA$.
By Lemma \ref{lem5:cardinalityA}, $\#\Os_{\ell'} G_q= q(q-1)(q+1)$,  $\#\Os_{\ell''} G_q=\#\Os_{\ell'''} G_q= \frac{1}{2}(q-1)(q+1)$.
\end{proof}

\section{Open problems for $\EnG$-lines and their solutions for $5\le q\le37$ and $q=64$}\label{sec:classification}

We introduce sets $Q^{(\xi)}_\bullet$ of $q$ values with the natural subscripts ``$\mathrm{od}$'' and ``$\mathrm{ev}$".
\begin{align*}
        &Q^{(0)}_\t{od}=\{9,27\},~Q^{(1)}_\mathrm{od}=\{7,13,19,25,31,37\},~Q^{(-1)}_\mathrm{od}=\{5,11,17,23,29\};\db\\
             &Q_\t{ev}=\{8,16,32,64\}.
    \end{align*}
Theorem \ref{th8_MAGMA_odd}  has been proved by an  exhaustive computer search using the symbol calculation system Magma \cite{Magma}.

\begin{thm} \label{th8_MAGMA_odd}For $q\in Q^{(1)}_\mathrm{od}\cup Q^{(-1)}_\mathrm{od}\cup Q^{(0)}_\mathrm{od}$ and $q\in Q_\t{\emph{ev}}$, all the results of Sections \emph{\ref{sec_mainres}--\ref{sec:orbEA}} are confirmed by computer search. In addition, the following holds, see Notation \emph{\ref{notation_1}}.
\begin{description}
  \item[(i)]
Let $q\equiv\xi\pmod3$,~$\xi\in\{1,-1,0\}$.  Let $q\in Q^{(1)}_\mathrm{od}\cup Q^{(-1)}_\mathrm{od}\cup Q^{(0)}_\mathrm{od}$ be odd. Then we have the following:

The total number of $\EnG$-line orbits is  $L_{\EnG\Sigma}^{(\xi)\mathrm{od}}=2q-3+\xi$.

The total number of line orbits in $\PG(3,q)$ is $L_\Sigma^{(\xi)}=2q+7+\xi$.

 Under $G_q$, for $\EnG$-lines with $\xi\in\{1,-1,0\}$, there are\\
 $
     \begin{array}{lcl}
       (2q-6-4.5\xi^2-0.5\xi)/3 &\t{orbits of length}& (q^3-q)/4, \\
       q-1&\t{orbits of length}&(q^3-q)/2,\\
       (q-\xi)/3&\t{orbits of length}&q^3-q.
     \end{array}
     $\\
     In addition, for $q\in Q^{(1)}_\t{\emph{od}}$, there are\\
     $
     \begin{array}{lcl}
       1 &\t{orbit of length}&(q^3-q)/12, \\
       2&\t{orbits of length}&(q^3-q)/3.
     \end{array}
     $

  \item[(ii)]
  Let $q\equiv\xi\pmod3$, $\xi\in\{1,-1\}$.  Let $q\in Q_\t{\emph{ev}}$ be even. Then we have the following:

The total number of $\EnG$-line orbits is  $L_{\EnG\Sigma}^{(\xi)\mathrm{ev}}=2q-2+\xi$.

The total number of line orbits in $\PG(3,q)$ is $L_\Sigma^{(\xi)}=2q+7+\xi$.

    Under $G_q$, for $\EnG$-lines, there are

     $2+\xi$  orbits of length $(q^3-q)/(2+\xi)$;

     $2q-4$ orbits of length $(q^3-q)/2$.
 \end{description}
\end{thm}

\begin{conjecture}
The results of Theorem \emph{\ref{th8_MAGMA_odd}} hold for all $q\ge5$ with the corresponding parity and $\xi$ value.
\end{conjecture}

\noindent \textbf{Open problems.}
Find the number, sizes and the structures of orbits of the class $\O_6=\O_\EnG$ (i.e. external lines, other than chord, not in a $\Gamma$-plane).
Prove the corresponding results of Theorem \ref{th8_MAGMA_odd} for all $q\ge5$.

\section*{Acknowledgments}
 The research of S. Marcugini, and F.~Pambianco was supported in part by the Italian National Group for Algebraic and Geometric Structures and their Applications (GNSAGA - INDAM) and by University of Perugia (Project: Curve, codici e configurazioni di punti, Base Research Fund 2018).

\end{document}